\documentclass[11pt]{amsart} 

\usepackage{cite}
\usepackage[utf8]{inputenc}
\usepackage[T1]{fontenc} 
\usepackage[english]{babel} 
\usepackage{amsmath,amsfonts,amsthm, mathrsfs} 
\usepackage{graphics}
\usepackage{hyperref}
\usepackage{tikz-cd}
\usepackage{adjustbox}
\usepackage{mathdots}

\theoremstyle{plain}
\newtheorem{theorem}{Theorem}[section]
\newtheorem{corollary}[theorem]{Corollary}
\newtheorem{lemma}[theorem]{Lemma}
\newtheorem{proposition}[theorem]{Proposition}

\theoremstyle{definition}

\newtheorem{remark}[theorem]{Remark}
\numberwithin{equation}{section}
\DeclareMathOperator{\ch}{ch}
\DeclareMathOperator{\GL}{GL}
\newcommand{\GLq}[2][q]{\GL_{{#2}}(\mathbf{F}_{{#1}})}
\DeclareMathOperator{\Sp}{Sp}
\newcommand{\Spq}[2][q]{\Sp_{{#2}}(\mathbf{F}_{{#1}})}

\DeclareMathOperator{\U}{U}
\DeclareMathOperator{\Tr}{Tr}
\DeclareMathOperator{\sgn}{sgn}
\DeclareMathOperator{\Ind}{Ind}
\DeclareMathOperator{\Res}{Res}
\DeclareMathOperator{\Av}{Av}

\title[A characteristic map for $\GLq{2n}/\Spq{2n}$]{	
A characteristic map for the symmetric space of symplectic forms over a finite field
}

\author{Jimmy He}
\address{Department of Mathematics, Stanford University, Stanford, CA  94305}
\email{jimmyhe@stanford.edu}

\begin{document}
\begin{abstract}
The characteristic map for the symmetric group is an isomorphism relating the representation theory of the symmetric group to symmetric functions. An analogous isomorphism is constructed for the symmetric space of symplectic forms over a finite field, with the spherical functions being sent to Macdonald polynomials with parameters $(q,q^2)$. An analogue of parabolic induction is interpreted as a certain multiplication of symmetric functions. Applications are given to Schur-positivity of skew Macdonald polynomials with parameters $(q,q^2)$ as well as combinatorial formulas for spherical function values.
\end{abstract}
\maketitle
\section{Introduction}
\subsection{Motivation}
There is a close connection between the representation theory of certain groups and symmetric functions, of which the most well-known and classical is for the symmetric group. In particular, an isomorphism between the class functions on the symmetric groups and the ring of symmetric functions can be constructed. Here the graded multiplication is given by Young induction, and the irreducible characters are sent to the Schur functions.

In Macdonald's classic book on symmetric functions \cite{M95}, two extensions of this are given, one for $\GLq{n}$ (originally due to Green \cite{G55}), and one for the Gelfand pair $S_{2n}/B_n$, where $B_n$ denotes the hyperoctahedral group (originally due to Stembridge \cite{S92}, although a similar connection to the symmetric space $\GL(n,\mathbf{R})/\operatorname{O}(n,\mathbf{R})$ was noticed by James \cite{J61}). A characteristic map can be constructed in both cases and they have applications to computing character and spherical function values.

This paper develops an analogous theory for $\GLq{2n}/\Spq{2n}$ and some applications to Macdonald polynomials are given. The symmetric space $\GLq{2n}/\Spq{2n}$ is a natural $q$-analogue of the Gelfand pair $S_{2n}/B_n$; the former can be seen as the Weyl group version of the latter. Work of Bannai, Kawanaka and Song \cite{BKS90} gave a formula for the spherical functions in terms of so-called basic functions, which are the analogues of the Deligne-Lusztig characters in this setting. Already, coefficients related to the Macdonald polynomials with parameters $(q,q^2)$ appeared, and so it is natural to seek an analogue of the characteristic map for $\GLq{2n}/\Spq{2n}$.

The original motivation to seek this construction was to analyze a random walk on the symmetric space by seeking a combinatorial formula for the spherical functions. A more probabilistic proof was subsequently found and the analysis of the Markov chain can be found in \cite{H19}.

\subsection{Main results}
The main contribution is to construct a characteristic map
\begin{equation*}
    \ch:\bigoplus_n \mathbf{C}[\Spq{2n}\backslash \GLq{2n}/\Spq{2n}]\to \bigotimes \Lambda
\end{equation*}
from the space of bi-invariant functions to a ring of symmetric functions and establish some basic properties. In particular, the spherical functions are shown to map to Macdonald polynomials with parameters $(q,q^2)$. Unfortunately, the multiplication on the ring of symmetric functions does not seem to have an easy representation-theoretic interpretation, and in particular is not given by parabolic induction in the sense defined by Grojnowski \cite{G92} (a previous version of this paper claimed to show this). Nevertheless, a "mixed product" can be defined using parabolic induction which takes a bi-invariant function on $\GLq{2n}$ and a class function on $\GLq{m}$ and produces a bi-invariant function on $\GLq{2(n+m)}$ and this (almost) corresponds to the product of symmetric functions under the characteristic map. It would be interesting to find a representation-theoretic interpretation of symmetric function multiplication under the characteristic map.

This result should be seen as an application of the results in \cite{BKS90} and \cite{H01t} to develop a framework to translate between the representation theory of $\GLq{2n}/\Spq{2n}$ and the theory of symmetric functions. As examples of this, two applications are given.

The first is an application to Schur-positivity of certain (unmodified) Macdonald polynomials. If $P_{\lambda/\mu}(x;q,t)$ denotes the skew Macdonald polynomials, let $C_{\lambda/\mu}^\nu(q,t)$ be defined by
\begin{equation*}
    C_{\lambda/\mu}^\nu(q,t)=\langle P_{\lambda/\mu}(q,t), s_\nu\rangle
\end{equation*}
where the inner product is the standard Hall inner product. These are the coefficients of the expansion of $P_{\lambda/\mu}$ in terms of Schur functions. They can be considered a deformed Littlewood-Richardson coefficient, which can be defined as $\langle s_{\lambda/\mu},s_\nu\rangle$.

The following theorem is proven using the characteristic map.
\begin{theorem}
\label{thm: positivity}
Let $q=p^e$ denote an odd prime power. Then for any partitions $\mu,\nu,\lambda$
\begin{equation*}
    C_{\lambda/\mu}^\nu(q,q^2)\geq 0.
\end{equation*}
\end{theorem}

This theorem is related to a conjecture of Haglund \cite{Y12}, which states that
\begin{equation*}
   \frac{\langle J_\lambda(q,q^k),s_\mu\rangle}{(1-q)^{|\lambda|}}\in \mathbf{N}[q]
\end{equation*}
and provides further evidence of this conjecture. In particular, taking $\mu=0$ in Theorem \ref{thm: positivity} shows 
\begin{equation*}
    \frac{\langle J_\lambda(q,q^2),s_\mu\rangle}{(1-q)^{|\lambda|}}\geq 0
\end{equation*}
when $q$ is an odd prime power.

The Schur expansion of integral Macdonald polynomials has been previously studied by Yoo who gave combinatorial formulas for the expansion coefficients in some special cases showing that they were polynomials in $q$ with positive integer coefficients \cite{Y12,Y15}. Some related coefficients in the Jack case when $q\to 1$ case was studied in \cite{AHW18}. The Schur-positivity proven is distinct from the well-known positivity result proven by Haiman \cite{H01b}, which deals with the modified Macdonald polynomials.

The second application is to combinatorial formulas for the values of spherical functions on certain double cosets. Although there are formulas in \cite{BKS90} already for all values of the spherical functions, they are alternating and so are unsuitable for asymptotic analysis. Asymptotic analysis of spherical functions appears in the study of Markov chains on Gelfand pairs, see \cite{D88, CST08} for examples. 

\subsection{Related work}
In addition to the examples already mentioned for $S_n$, $\GLq{n}$ and $S_{2n}/B_n$, there are further examples of characteristic maps appearing in the literature. In particular, Thiem and Vinroot constructed such a map for $\U_n(\mathbf{F}_{q^2})$ \cite{TV07} and in \cite{Metal12} a map is constructed between the supercharacters of the unipotent upper triangular matrices over a finite field and symmetric functions in non-commuting variables.

The connections between the representation theory of $\GLq{2n}/\Spq{2n}$ and Macdonald polynomials are not new and were already noticed in \cite{BKS90}. A further connection was made by Shoji and Sorlin between a related space and the modified Kostka polynomials which are the change of basis from the Hall-Littlewood polynomials to the Schur functions \cite{SS14a}.

As there is a rational parabolic for the Levi subgroup used in the parabolic induction, elementary proofs have been given for any necessary results. The statements and proofs were all inspired by a more general construction of Grojnowski \cite{G92} for parabolic induction of $l$-adic sheaves and subsequently studied by Henderson \cite{H01t}. Some of this work was later extended in the work of Shoji and Sorlin \cite{SS13,SS14a,SS14b} which can also be found in the survey \cite{S16}. The more general construction has the benefit of working even without a rational parabolic subgroup, and in particular the construction of a characteristic map for $\U_{2n}(\mathbf{F}_{q^2})/\Spq{2n}$ should be relatively straightforward although the combinatorial application in Section \ref{sec: LR coeff} would not extend.

The $q\to 1$ case was previously studied by Bergeron and Garsia \cite{BG92} (see also \cite{M95}), where the Gelfand pair $S_{2n}/B_n$ played a similar role. In this case the multiplication on symmetric functions has a representation-theoretic interpretation giving results on the structure coefficients of Jack polynomials for $\alpha=2$. The lack of representation-theoretic interpretation in the $q$-deformed case means that their ideas cannot be directly applied here.

Macdonald polynomials with parameters $(q,q^2)$ also appeared previously in the study of quantum symmetric spaces \cite{N96}, and it would be interesting to see if there is a connection.

\subsection{Outline}
The paper is organized as follows. In Section \ref{section: preliminaries}, notation and preliminary background is reviewed. Section \ref{sec: induction for functions} develops the theory of parabolic induction for bi-invariant functions in an elementary way. In Section \ref{sec:characteristic map}, the characteristic map for $\GLq{2n}/\Spq{2n}$ is constructed. Section \ref{sec: LR coeff} gives an application to positivity and vanishing of the Schur expansion of skew Macdonald polynomials with parameters $(q,q^2)$ and Section \ref{sec:computation of spherical function values} gives an application to computing spherical function values. Section \ref{sec: technical results} explains how parabolic induction for functions is a special case of parabolic induction of sheaves on symmetric spaces.

Section \ref{sec: technical results} is the only section which requires any familiarity with character sheaves and the rest of the paper is independent of it.

\section{Preliminaries}
\label{section: preliminaries}
In this section, some needed background is reviewed and the notation and conventions that are used are explained. The notation and background on character sheaves needed is left for Section \ref{sec: technical results}.

\subsection{Notation}
If $f(q)$ is a rational function in $q$, define $f(q)_{q\mapsto q^2}:=f(q^2)$. For example,
\begin{equation*}
|\GLq{n}|_{q\mapsto q^2}=\prod_{i=0}^{n-1}(q^{2n}-q^{2i}).
\end{equation*}
For a symmetric function $f$ with rational coefficients in $q$ when written in terms of $p_\mu$, write $f_{q\mapsto q^2}$ to denote the symmetric function obtained by replacing $q$ with $q^2$ in each coefficient.

Let $G$ be a finite group. If $S\subseteq G$ is some subset, $I_S$ will denote the indicator function for that set. If $H\subseteq G$ is a subgroup, let $\mathbf{C}[G]^G$ denote the set of class functions on $G$ and let $\mathbf{C}[H\backslash G/H]$ denote the set of $H$ bi-invariant functions on $G$. These carry the usual inner product $\langle f,g\rangle=\sum _{x\in G}f(x)\overline{g(x)}$.

If $S,H$ are subgroups of any group $G$, let $H_S=H\cap S$. Given an element $x$ of some finite field extension of $\mathbf{F}_q$, let $f_x$ denote its minimal polynomial.

\subsection{Macdonald polynomials}
For details on Macdonald polynomials including a construction and proofs, see \cite{M95}. Let $\Lambda$ denote the ring of symmetric functions. Consider the inner product on the ring of symmetric functions defined by
\begin{equation*}
\langle p_\lambda,p_\mu\rangle_{q,t}=\delta_{\lambda\mu}z_\lambda\prod \frac{q^{\lambda_i}-1}{t^{\lambda_i}-1},
\end{equation*}
where $z_\lambda=\prod m_i(\lambda)!i^{m_i(\lambda)}$ and $m_i(\lambda)$ denotes the number of parts of size $i$ in $\lambda$. When the context is clear, the dependence on $(q,t)$ may be dropped.

This specializes to the Hall inner product when $q=t$ and setting $q=t^{\alpha}$ and taking a limit gives the Jack polynomial inner product. The \emph{Macdonald polynomials} $P_\lambda(x;q,t)$ (indexed by partitions $\lambda$) are defined by the fact that they are orthogonal with respect to this inner product, and the change of basis to the monomial basis is upper triangular with $1$ along the diagonal. When $q=0$, the Macdonald polynomials are known as \emph{Hall-Littlewood polynomials}, and are written $P_\lambda(x;t)=P_\lambda(x;0,t)$.

Define
\begin{equation*}
c_\lambda(q,t):=\prod _{s\in \lambda}(1-q^{a(s)}t^{l(s)+1}),
\end{equation*}
where $a(s)$ and $l(s)$ denote the arm and leg lengths respectively (so $a(s)+l(s)+1=h(s)$). Similarly define
\begin{equation*}
c'_\lambda(q,t):=\prod _{s\in \lambda}(1-q^{a(s)+1}t^{l(s)}).
\end{equation*}
The dual basis to the $P_\lambda(x;q,t)$ under $\langle ,\rangle_{q,t}$ are denoted $Q_\lambda(x;q,t)$ and satisfy
\begin{equation*}
    Q_\lambda(x;q,t)=\frac{c_\lambda(q,t)}{c'_\lambda(q,t)}P_\lambda(x;q,t).
\end{equation*}

The Macdonald polynomials have an integral form 
\begin{equation*}
    J_\lambda(x;q,t)=c_\lambda(q,t)P_\lambda(x;q,t)=c'_\lambda(q,t)Q_\lambda(x;q,t).
\end{equation*}
Then $\langle J_\lambda,J_\lambda\rangle=c_\lambda(q,t)c_\lambda'(q,t)$.

The symmetric functions $J_\lambda$ can be thought of as a deformation of the Jack polynomials, which are given by taking $q=t^{\alpha}$ and sending $t\to 1$ after dividing by $(1-t)^n$. 

There are also various homomorphisms defined on $\Lambda$, defined through their action on $p_n$. These are also commonly written using plethystic notation but to match the notation in Macdonald this is avoided.

Let $\omega:\Lambda\to \Lambda$ denote the involution taking $p_n$ to $(-1)^{n-1}p_n$ (and extending to make it an algebra homomorphism). Let $\omega_{q,t}:\Lambda\to\Lambda$ denote the involution defined by
\begin{equation*}
    \omega_{q,t}p_n=(-1)^{n-1}\frac{q^n-1}{t^n-1}p_n.
\end{equation*}
The involution $\omega$ interacts well with Schur functions, with
\begin{equation*}
    \omega s_\lambda=s_{\lambda'}
\end{equation*}
while the involution $\omega_{q,t}$ satisfies
\begin{equation*}
\begin{split}
    \omega_{q,t}P_\lambda(x;q,t)&=Q_{\lambda'}(x;t,q)
    \\\omega_{q,t}Q_\lambda(x;q,t)&=P_{\lambda'}(x;t,q).
\end{split}
\end{equation*}

\subsection{Linear Algebraic groups}
The point of view taken is to view the finite groups of interest as rational points of an algebraic group over $\overline{\mathbf{F}_q}$. Thus, the conventions and notation may differ slightly from more classical sources such as Carter \cite{C93} which work directly over the finite field. The definitions and conventions taken mostly follow Digne and Michel \cite{DM91}.

In general, linear algebraic groups $G$ will be defined over $\overline{\mathbf{F}_q}$. Let $F$ denote the Frobenius endomorphism, which will always be the one taking the matrix $(x_{ij})$ to $(x_{ij}^q)$. Then $G(\mathbf{F}_q)$ or $G^F$ will denote the $\mathbf{F}_q$ points.

A \emph{torus} is a group which is isomorphic to $(\overline{\mathbf{F}_q}^*)^n$ for some $n$. Note that the maximal tori contained in $\GL_n$ are precisely the subgroups which are conjugate to the standard maximal torus, which consists of diagonal matrices.

A \emph{Levi subgroup} $L\subseteq G$ is a subgroup that is the centralizer of some torus $T$. The Levi subgroups of $\GL_n$ are all of the form $\prod \GL_{n_i}$ for $\sum n_i=n$. A \emph{Borel subgroup} in $\GL_n$ is some subgroup conjugate to the subgroup of upper triangular matrices which is the standard Borel subgroup.

A \emph{parabolic subgroup} is a subgroup containing a Borel subgroup. In $\GL_n$ these subgroups are conjugate to some subgroup of block upper-triangular matrices which are the standard parabolics. All parabolic subgroups have a decomposition $P=LU$ where $U$ is the \emph{unipotent radical} of $P$ and $L$ is a choice of some Levi subgroup, called the \emph{Levi factor} of $P$. In $\GL_n$, for a standard parabolic consisting of block upper-triangular matrices, the unipotent radical consists of matrices which are the identity in each block and arbitrary above the diagonal blocks. As an example, in $\GL_4$
\begin{equation*}
\setlength{\arraycolsep}{3pt}
    P=\left\{\left(\begin{array}{cccc}
         *&*&*&*\\
         *&*&*&*\\
         0&0&*&*\\
         0&0&*&*
    \end{array}\right)\right\}, \quad
    U=\left\{\left(\begin{array}{cccc}
         1&0&*&*\\
         0&1&*&*\\
         0&0&1&0\\
         0&0&0&1
    \end{array}\right)\right\},\quad
    L=\left\{\left(\begin{array}{cccc}
         *&*&0&0\\
         *&*&0&0\\
         0&0&*&*\\
         0&0&*&*
    \end{array}\right)\right\}
\end{equation*}
are an example of a parabolic subgroup and the unipotent radical and Levi subgroup respectively (the $*$'s denote arbitrary entries although the matrices must be invertible). As the unipotent radical is normal, every parabolic subgroup has a canonical projection $P\to L$ which for a standard parabolic in $\GL_n$ amounts to replacing the blocks on the diagonal of a matrix in $P$ with identity matrices. In general the projection of $x\in P$ will be denoted by $\overline{x}$. As Borel subgroups are parabolic, the same applies with $L$ being a maximal torus.

Say that a subgroup $H\subseteq G$ is \emph{rational} (or $\mathbf{F}_q$-rational) if it is stable under $F$. Any rational subgroup $H$ defines a subgroup $H^F\subseteq G^F$. Conversely, a subgroup of $G^F$ is always algebraic and so defines a rational subgroup $H\subseteq G$ by base change. The standard Levi and parabolic subgroups are rational.

In $\GL_n$, the rational points of rational maximal tori are of the form $\prod \mathbf{F}_{q^i}$ and two rational maximal tori are conjugate under $GL_n^F$ if and only if their rational points are isomorphic. 

\subsection{Representation theory of $\GLq{n}$}
To fix notation, the representation theory of $\GLq{n}$ is briefly reviewed. The representation theory of $\GLq{n}$ was originally developed by Green in \cite{G55} but this section follows  the conventions in \cite{M95}. Let $\mathbf{M}$ denote the group of units of $\overline{\mathbf{F}}_q$ and let $\mathbf{M}_n$ denote the fixed points of $F^n$ with $F$ the Frobenius endomorphism $F(x)=x^q$. Note that $\mathbf{M}_n$ may be identified with $\mathbf{F}_{q^{n}}^*$. Let $\mathbf{L}$ be the character group of the inverse limit of the $\mathbf{M}_n$ with norm maps between them. The Frobenius endomorphism $F$ acts on $\mathbf{L}$ in a natural manner so let $\mathbf{L}_n$ denote the $F^n$ fixed points in $\mathbf{L}$, and note there is a natural pairing of $\mathbf{L}_n$ with $\mathbf{M}_n$ for each $n$ (but these pairings are not consistent).

The $F$-orbits of $\mathbf{M}$ can be viewed as irreducible polynomials over $\mathbf{F}_q$ under $O\mapsto \prod_{\alpha\in O}(x-\alpha)$. Denote by $O(\mathbf{M})$ and $O(\mathbf{L})$ the $F$-orbits in $\mathbf{M}$ and $\mathbf{L}$ respectively. Use $\mathcal{P}$ to denote the set of partitions. Then the conjugacy classes of $\GLq{n}$, denoted $C_\mu$, are indexed by partition-valued functions $\mu:O(\mathbf{M})\rightarrow \mathcal{P}$ such that 
\begin{align*}
\|\mu\|:=\sum _{f\in O(\mathbf{M})}d(f)|\mu(f)|=n,
\end{align*}
where $d(f)$ denotes the degree of $f$. This is because $\mu$ contains the information necessary to construct the Jordan canonical form. That is, given $\mu$, construct a matrix in $\GL_{n}(\overline{\mathbf{F}_q})$ in Jordan form by taking for each orbit $f\in O(\mathbf{M})$, $l(\mu(f))$ blocks, of sizes $\mu(f)_i$, for each root of $f$. The resulting matrix has $d(f)$ blocks of size $\mu(f)_i$ for each $f$ and $i$, and adding this all up gives $\|\mu\|=n$.

For example, the partition-valued function corresponding to the set of transvections, which have Jordan form
\begin{align*}
\left(\begin{array}{ccccc}
1&1&0&\dots&0\\
0&1&0&\dots&0\\
0&0&1&\dots&0\\
\vdots&\vdots&\vdots&\ddots&\vdots\\
0&0&0&\dots&1
\end{array}\right),
\end{align*}
correspond to the partition-valued function $\mu$ with $\mu(f_1)=(21^{n-2})$ and $\mu(f)=0$ for $f\neq f_1$. Use $q_f$ to denote $q^{d(f)}$. There is a formula for the sizes of conjugacy classes given by
\begin{equation*}
|C_\mu|=\frac{|\GLq{n}|}{a_\mu(q)}
\end{equation*}
where
\begin{equation*}
a_\mu(q)=q^n\prod _{f\in O(\mathbf{M})}q_f^{2n(\mu(f))}\prod_{i\geq 1}\prod _{j=1}^{m_i(\mu(f))}(1-q_f^{-j}),
\end{equation*}
with $n(\lambda)=\sum (i-1)\lambda_i$.

Similarly, the irreducible characters of $\GLq{n}$, denoted $\chi_\lambda$, are indexed by functions $\lambda:O(\mathbf{L})\rightarrow \mathcal{P}$ such that
\begin{equation*}
\|\lambda\|:=\sum _{\varphi\in O(\mathbf{L})}d(\varphi)|\lambda(\varphi)|=n,
\end{equation*}
where $d(\varphi)$ denotes the size of the orbit $\varphi$. The dimension of the irreducible representation corresponding to $\lambda$ is given by
\begin{equation*}
d_{\lambda}=\psi_n(q)\prod _{\varphi\in O(\mathbf{L})}q_\varphi^{n(\lambda(\varphi)')}H_{\lambda(\varphi)}(q_\varphi)^{-1},
\end{equation*}
where $\psi_n(q)=\prod _{i=1}^n(q^i-1)$, $q_\varphi=q^{d(\varphi)}$ and $H_\lambda(t)=\prod _{x\in \lambda}(t^{h(x)}-1)$, $h(x)$ denoting the hook length. Note that with this convention, the trivial representation corresponds to the partition-valued function $\lambda(\chi_{1})=(1^n)$ ($\chi_1$ being the trivial character) and $0$ otherwise (this differs from the usual convention for $S_n$, where the trivial representation corresponds to the partition $(n)$).

As a matter of convention, $\mu$ will always be used to denote partition-valued functions $O(\mathbf{M})\to \mathcal{P}$ while $\lambda$ will be used to denote partition-valued functions $O(\mathbf{L})\to\mathcal{P}$. In general, if there is some expression involving a partition $\mu$, $F(\mu)$, with $F(0)=1$, then the same expression with $\mu:O(\mathbf{M})\rightarrow \mathcal{P}$ will be defined as the product $\prod _{f\in O(\mathbf{M})}F(\mu(f))$. Thus, write
\begin{equation*}
z_\mu=\prod _{f\in O(\mathbf{M})}z_{\mu(f)}.
\end{equation*}
If the expression contains $q$, in each factor it should be replaced by $q_f$. A similar convention is used for $\lambda:O(\mathbf{L})\to\mathcal{P}$.

The \emph{Deligne-Lusztig characters} (also called basic functions for $\GLq{n}$) give another basis for the space of class functions on $\GLq{n}$. For a more detailed overview of Deligne-Lusztig characters, including their construction and properties, see the book of Carter \cite{C93}.

Given any rational maximal torus $T\subseteq \GL_n$, a (virtual) character $\zeta_T^{\GL_n}(\cdot|\theta)$ of $\GLq{n}$ associated to some irreducible character $\theta$ of $T^F$ can be constructed. The character constructed depends only on the $\GLq{n}$-conjugacy class of the pair $(T,\theta)$. 

Deligne-Lusztig characters are rational functions in $q$ in the following sense. Given an element $g\in \GLq{n}$ with Jordan decomposition $g=su$ ($s$ semisimple and $u$ unipotent), the Deligne-Lusztig characters can be computed as
\begin{equation*}
\zeta_T^{\GL_n}(g|\theta)=\sum _{\substack{x\in (\GL_n/Z(s))^F \\ xsx^{-1}\in T^F}}\theta(xsx^{-1})Q_{x^{-1}Tx}^{Z(s)}(u),
\end{equation*}
where $Q_{x^{-1}Tx}^{Z(s)}(u)$ is a rational function of $q$, known as the \emph{Green function}. The Green functions can be computed as
\begin{equation*}
Q_{x^{-1}Tx}^{Z(s)}(u)=\prod _{f\in O(\mathbf{M})}Q_{\gamma(f)}^{\mu(f)}(q_f),
\end{equation*}
where $Q_\rho^\mu(q)$ denotes the \emph{Green polynomials}, $\mu$ is a partition valued function indexing the conjugacy class of $g$ and $\gamma(f)$ is the partition given by taking $s\in x^{-1}T^Fx\cong \prod \mathbf{M}_{k_i}$ and including as parts the $k_i/d(f)$ for which $f$ kills $s$ restricted to $\mathbf{M}_{k_i}$. Since isomorphic maximal tori in $\GLq{n}$ are always conjugates, the tori $x^{-1}T^Fx$ for $x\in (\GL_n/Z(s))^F$ are exactly those containing $s$ up to conjugation by $Z(s)^F$, which in turn are of the form $\prod \mathbf{M}_{\gamma_i(f)d(f)}$ for $\gamma$ a partition valued function such that $|\gamma(f)|=|\mu(f)|$ for all $f\in O(\mathbf{M})$. This means that the Deligne-Lusztig character can be written
\begin{equation}
\zeta_T^{\GL_n}(g|\theta)=\sum _{t\in T,|\gamma_t(f)|=|\mu(f)|}\theta(t)\prod _{f\in O(\mathbf{M})}Q_{\gamma_t(f)}^{\mu(f)}(q_f),
\end{equation}
where $\gamma_t$ denotes the partition valued function corresponding to the torus $x^{-1}T^Fx$ for $t=xsx^{-1}$ and $s$ is the semisimple element as described above.

The Green polynomials give the change of basis from power sum to Hall-Littlewood polynomials and so satisfy
\begin{equation*}
p_\rho(x)=\sum _{\mu}Q_\rho^\mu(t)t^{-n(\mu)}P_\mu(x;t^{-1})
\end{equation*}
for a formal parameter $t$. 

The Deligne-Lusztig characters are invariant under conjugation of $(T,\theta)$ by $G^F$. There is a correspondence between $G^F$ orbits of pairs $(T,\theta)$ and functions $\lambda:O(\mathbf{L})\to\mathcal{P}$ with $\|\lambda\|=n$. Call the function $\lambda$ the \emph{combinatorial data} associated to $(T,\theta)$.

The correspondence is as follows. Given a torus $T$ with rational points isomorphic to $\prod \mathbf{M}_{k_i}$, and a character $\theta$ of $T^F$, consider the partition-valued function sending $\varphi$ to the partition with parts $k_i/d(\varphi)$ for all $i$ such that $\theta$ restricted to $\mathbf{M}_{k_i}$ lies in the orbit $\varphi$. Conversely, given $\lambda$, $T$ can be constructed as a torus with rational points isomorphic to $\prod_{\varphi,i} \mathbf{M}_{\lambda(\varphi)_i d(\varphi)}$ and $\theta$ is given by picking for each factor $\mathbf{M}_{\lambda(\varphi)_i d(\varphi)}$ an element of the orbit $\varphi$ (which determines $T$ and $\theta$ up to conjugacy by $G^F$).

Given some $\lambda:O(\mathbf{L})\rightarrow \mathcal{P}$, consider the Levi subgroup $L_\lambda$ whose rational points are given by
\begin{equation*}
\prod_{\varphi\in O(\mathbf{L})} \GL_{|\lambda(\varphi)|}(\mathbf{F}_{q^{d(\varphi)}})\subseteq \GLq{n},
\end{equation*}
with Weyl group $W(\lambda)=\prod _{\varphi\in O(\mathbf{L})}S_{|\lambda(\varphi)|}$. To each Weyl group element $w$ there is an associated partition-valued function sending $\varphi$ to the cycle type of $w(\varphi)$. Let $(T_w,\theta_{w})$ denote the torus and character associated to the combinatorial data defined by $w$. Then there is a formula relating the Deligne-Lusztig characters to the irreducible characters of $\GLq{n}$ given by
\begin{equation*}
    \chi_\lambda(g)=\frac{(-1)^{n-|\lambda|}}{|W(\lambda)|}\sum_{w\in W(\lambda)} \prod _{\varphi\in O(\mathbf{L})}\chi^{Sym}_{\lambda'(\varphi)}(w(\varphi))\zeta_{T_w}^{\GL_n}(g|\theta_{w}),
\end{equation*}
where the $\chi^{Sym}_\lambda$ are the irreducible characters of the symmetric group. Here $|\lambda|=\sum _\varphi |\lambda(\varphi)|$.

\subsection{The symmetric space $\GLq{2n}/\Spq{2n}$}
Let
\begin{equation*}
J=\left(\begin{array}{cccccc}
& & & & & -1\\
& & & & \iddots& \\
& & &-1 & & \\
& & 1& & & \\
 & \iddots & & & &\\
 1&  & & & & \\
\end{array}\right)
\end{equation*}
define the standard symplectic form (where any empty space in the matrix is $0$) and define the involution 
\begin{equation*}
    \iota(X)=-J(X^T)^{-1}J
\end{equation*}
of $\GL_{2n}$. This choice is made for convenience so that the upper triangular matrices are stable under $\iota$. This clearly commutes with the Frobenius map $F$. Then $\Sp_{2n}$ denotes the subgroup of $\GL_{2n}$ fixed by $\iota$.

Now for any $\iota$-stable subgroup $S$, let $S^\iota$ denote the subgroup of $\iota$ fixed points and $S^{-\iota}$ denote the set of $\iota$-\emph{split} elements, which are elements $s\in S$ such that $\iota(s)=s^{-1}$.

A \emph{Gelfand pair} is a (finite) group $G$, with a subgroup $H\subseteq G$ such that inducing the trivial representation from $H$ to $G$ gives a multiplicity-free representation. For a Gelfand pair $G/H$, any representation $\rho$ of $G$ has either no non-zero $H$-fixed vectors, or a $1$-dimensional subspace fixed pointwise by $H$. Say that $\rho$ is a \emph{spherical representation} if it has an $H$-fixed vector, and define the corresponding \emph{spherical function} to be $\phi(g)=\langle v_\rho,\rho(g)v_\rho\rangle$, where $v_\rho$ is a unit $H$-fixed vector. 

The spherical functions can also be computed by averaging characters over $H$. That is, $\phi(g)=|H|^{-1}\sum _{h\in H}\chi(hg)$ for $\chi$ the character of $\rho$ (if $\rho$ is not a spherical representation, then this average is $0$). The spherical functions are the replacement for characters of a group and in particular, they form a basis for the space of bi-invariant functions on $G$.

Now $\GLq{n}/\Spq{n}$ is a Gelfand pair, and the spherical functions, denoted $\phi_\lambda$, are indexed by partition-valued function $\lambda:O(\mathbf{L})\to\mathcal{P}$ with $\|\lambda\|=n$ (see \cite{BKS90}, although note a different convention is used in this paper so all partitions labeling representations are transposed). For a partition $\lambda$, let $\lambda\cup\lambda$ denote the partition which contains every part of $\lambda$ twice. Then $\phi_\lambda$ is the spherical function corresponding to the representation with character $\chi_{\lambda\cup\lambda}$ of $\GLq{2n}$.

Let $M_\mu\in \GLq{n}$ denote a conjugacy class representative of $C_\mu$. Let $g_\mu\in \GLq{2n}$ denote the matrix acting only on the first $n$ coordinates by $M_\mu$. The $\Spq{2n}$-double cosets of are indexed by $\mu:O(\mathbf{M})\rightarrow\mathcal{P}$, with $\|\mu\|=n$, with the $g_\mu$ being double coset representatives.

Two key results from \cite{BKS90} are reproduced below. The first relates the sizes of double cosets in $\GLq{2n}/\Spq{2n}$ to the sizes of conjugacy classes in $\GLq{n}$.

\begin{proposition}[{\cite[Proposition 2.3.6]{BKS90}}]
\label{prop: double coset size}
Let $\mu:O(\mathbf{M})\rightarrow \mathcal{P}$ with $\|\mu\|=n$. Then
\begin{equation*}
|H^Fg_\mu H^F|=|H^F||C_\mu|_{q\mapsto q^2},
\end{equation*}
where $H^Fg_\mu H^F$ denotes the double coset indexed by $\mu$ in $\GLq{2n}$ and $C_\mu$ denotes the conjugacy class indexed by $\mu$ in $\GLq{n}$.
\end{proposition}
The second result gives a formula for the values of spherical functions in terms of Deligne-Lusztig characters on $\GLq{n}$.

The function $\zeta_T^{\GL_n}(\cdot|\theta)$ (or any other class function on $\GLq{n}$) may be turned into an $H^F$-bi-invariant function on $\GLq{2n}$ by defining 
\begin{equation*}
\zeta_T^{\GL_n}(g_\mu|\theta)=\zeta_{T}^{\GL_n}(M_\mu|\theta).
\end{equation*}
and extending to the double-coset. It turns out that the correct analogue for Deligne-Lusztig characters are given by the \emph{basic functions}
\begin{equation*}
    \zeta_{T}^{G/H}(\cdot|\theta)=\zeta_T^{\GL_n}(\cdot|\theta)_{q\mapsto q^2}.
\end{equation*}
Here the Deligne-Lusztig characters are viewed as rational functions in $q$ as described above.

\begin{remark}
It may seem a little strange to index the basic functions by maximal tori of $\GL_n$. It is more natural to view $(T,\theta)$ as a maximal torus of $G$ stable under $\iota$ along with a $(T\cap H)^F$ bi-invariant function $\theta$ of $T^F$. The basic functions are then indexed by such pairs up to $H^F$-conjugacy. The basic functions may be seen as parabolically induced from the torus $T$ although instead of working with representations, $l$-adic sheaves must be used. Since there is a correspondence between these two indexing sets, there is no harm in choosing to index with maximal tori of $\GL_n$.
\end{remark}

The following theorem relates the spherical functions to the basic functions.
\begin{theorem}[{\cite[Theorem 6.6.1]{BKS90}}]
\label{thm: spherical function formula}
Let $\lambda:O(\mathbf{L})\rightarrow \mathcal{P}$ with $\|\lambda\|=n$. Then
\begin{equation*}
\begin{split}
\phi_\lambda=&\frac{(-1)^{|\lambda|}}{|W(\lambda)|}\sum _{w\in W(\lambda)}\left(\prod _{\varphi\in O(\mathbf{L})}q_\varphi^{-n(\lambda(\varphi)')}c_{\lambda(\varphi)'}(q_\varphi^2,q_\varphi)d_{\lambda(\varphi)'}(w)(q_\varphi)\right)
\\&\qquad\qquad\qquad\qquad\times\sgn(w)\frac{\zeta_{T_w}^{G/H}(\cdot|\theta_w)}{|T_w|\zeta_{T_w}^{G/H}(1|\theta_w)},
\end{split}
\end{equation*}
where $c_\lambda(q,t)$ denotes the scaling from the two parameter Macdonald polynomials to their integral forms, $d_\lambda(w)(q)$ denotes the change of basis from power sum to two parameter Macdonald polynomials $P_\lambda(q^2,q)$ and $\sgn$ denotes the sign character of $W(\lambda)$.
\end{theorem}
Note that the convention used to label spherical representations in \cite{BKS90} is opposite the one used in this paper so all partitions are transposed.

\section{Induction and restriction}
\label{sec: induction for functions}
This section defines a bi-invariant version of the usual parabolic induction and restriction on reductive groups. Only the case when $G=\GL_{2n}$ will be needed, but all results except those concerning basic functions hold for any connected reductive group $G$ and involution $\iota$ (of algebraic groups). Only induction through rational parabolic subgroups will be considered.

Everything in this section follows from (and was inspired by) more general results on a bi-invariant parabolic induction defined for sheaves on symmetric spaces originally due to Grojnowski. These results can be found in \cite{G92,H01t,SS13,SS14a,SS14b}. The connection between the two induction procedures and how the results below follow from results for sheaves is explained in Section \ref{sec: technical results}. To keep the paper accessible and self-contained as much as possible, more elementary proofs of the needed results are given below.

\subsection{Tori, Borel subgroups and parabolic subgroups}
Let $G$ be a connected reductive group defined over $\mathbf{F}_q$ ($q$ odd) with Frobenius endomorphism $F$ and an involution $\iota$ of algebraic groups commuting with $F$. Let $H\subseteq G$ be the subgroup of fixed points of $\iota$.

Fix a rational $\iota$-stable pair $(T,B)$ of a maximal torus $T$ and Borel subgroup $B$ containing $T$. Say that a parabolic subgroup is \emph{standard} if it contains $B$. For convenience in later sections, if $G=\GL_n$, $T$ can be taken to be the set of diagonal matrices in $\GL_n$. 

All pairs of rational $\iota$-stable $(T,B)$ with $T\subseteq B$ are conjugate under $H^F$ (see \cite{SS13} for example). Furthermore, all rational $\iota$-stable Levi subgroups in an $\iota$-stable parabolic $P$ are conjugate under $H^F\cap U^F$ since they're conjugate by a unique element of the unipotent radical.

Since $\iota$ acts on $G$ and preserves $T$ and $B$, it preserves the simple roots and so determines an automorphism of the Dynkin diagram of $G$. Standard parabolic subgroups correspond to subsets of the Dynkin diagram and it's not hard to see that $\iota$-stable parabolics correspond to $\iota$-stable subsets. This gives an easy way to classify the Levi subgroups with $\iota$-stable rational parabolics.

In the case of $G=\GL_{2n}$ and $\iota$ fixing the symplectic group, the Dynkin diagram is $A_{2n-1}$ and the action of $\iota$ is given by reflecting along the middle. Then the $\iota$-stable parabolic subgroups correspond to subsets of the simple roots which are symmetric about the middle, consisting of pairs of the type $A_{n_i}$ and one of type $A_{2n_0-1}$ with $\sum n_i=n$. The corresponding Levi subgroups are then of the form $\GL_{2n_0}\times \prod \GL_{n_i}\times \GL_{n_i}$ where $\GL_{2n_0}$ corresponds to the middle component of the diagram (and $n_0=0$ if the middle vertex is not included) and the $\GL_{n_i}\times \GL_{n_i}$ factors correspond to the symmetric pairs of components.

Take $T\subseteq B$ to be the diagonal matrices and upper triangular matrices. Both are stable under $\iota$. Then the standard parabolic subgroups are exactly the block upper triangular matrices, and the $\iota$-stable ones are those whose block structure is symmetric about the anti-diagonal. 

\subsection{Bi-invariant parabolic induction}
Let $P$ be a rational $\iota$-stable parabolic subgroup with rational $\iota$-stable Levi factor $L$ and unipotent radical $U$. Then define a function
\begin{equation*}
    \Ind_{L\subseteq P}^{G/H}:\mathbf{C}[H_L^F\backslash L^F/H_L^F]\to \mathbf{C}[H^F\backslash G^F/H^F]
\end{equation*}
with the formula
\begin{equation*}
    \Ind_{L\subseteq P}^G(f)(x)=|H^F\cap P^F|^{-2}\sum _{\substack{h,h'\in H^F\\hxh'\in P^F}}f(\overline{hxh'}).
\end{equation*}
It is clear that this defines an $H^F$ bi-invariant function. This definition reduces to the standard definition of parabolic (or Harish-Chandra) induction when considering the group $G\times G$ and the involution switching the two factors.

\begin{remark}
The choice is made to work with bi-invariant functions on $G^F$ rather than invariant functions on $G^F/H^F$, which is the convention taken in \cite{BKS90} and \cite{H01t}. Of course these two points of view are completely equivalent and indeed the induction implicitly used in \cite{BKS90} can be viewed as $\Phi_G\circ \Ind_{L\subseteq P}^{G/H}\circ \Phi_L^{-1}$ where $\Phi_G(f)(xH^F)=f(x)$.
\end{remark}

When studying induction, many times it is convenient to work with standard parabolic subgroups. The next proposition shows that induction is invariant under conjugation by $H^F$ and so it always suffices to study induction through standard parabolics.

\begin{lemma}
For any $h_0\in H^F$ and $f\in \mathbf{C}[H_L^F\backslash L^F/H_L^F]$,
\begin{equation*}
    \Ind_{h_0Lh_0^{-1}\subseteq h_0Ph_0^{-1}}^{G/H}(^{h_0}f)=\Ind_{L\subseteq P}^{G/H}(f)
\end{equation*}
where $^{h_0}f$ is the function $^{h_0}f(x)=f(h_0^{-1}xh_0)$.
\end{lemma}
\begin{proof}
Note that the projection map $h_0Ph_0^{-1}\to h_0Lh_0^{-1}$ can be written as $x\mapsto h_0\overline{h_0^{-1}xh_0}h_0^{-1}$ where $\overline{x}$ denotes the projection for $P$. Then compute
\begin{equation*}
\begin{split}
    &\Ind_{h_0Lh_0^{-1}\subseteq h_0Ph_0^{-1}}^{G/H}(^{h_0}f)(x)
    \\=&|H^F\cap P^F|^{-2}\sum _{\substack{h,h'\in H^F\\hxh'\in h_0P^Fh_0^{-1}}}f(h_0^{-1}h_0\overline{h_0^{-1}hxh'h_0}h_0^{-1}h_0)
    \\=&|H^F\cap P^F|^{-2}\sum _{\substack{h,h'\in H^F\\hxh'\in P^F}}f(\overline{hxh'})
    \\=&\Ind_{L\subseteq P}^{G/H}(f)(x)
\end{split}
\end{equation*}
where $H\cap h_0Ph_0^{-1}=h_0(H\cap P)h_0^{-1}$.
\end{proof}

\begin{proposition}
\label{prop: transitivity for functions}
If $M\subseteq L$ are rational $\iota$-stable Levi subgroups with $Q\subseteq P$ rational $\iota$-stable parabolics with Levi factors $M,L$ respectively, then
\begin{equation*}
    \Ind_{M\subseteq Q}^{G/H}=\Ind_{L\subseteq P}^{G/H}\circ \Ind_{M\subseteq Q\cap L}^{L/H_L}.
\end{equation*}
\end{proposition}
\begin{proof}
This follows from computing both sides. In particular, note that the quotient map $P\to L$ is $H_L$ bi-equivariant and so
\begin{equation*}
\begin{split}
    &|H^F\cap P^F|^{-2}|H_L^F\cap (Q\cap L)^F|^{-2}\sum _{\substack{h_1,h_1'\in H^F\\h_1xh_1'\in P^F}}\sum _{\substack{h_2,h_2'\in H_L^F\\h_2\overline{h_1xh_1'}h_2'\in (Q\cap L)^F}}f(\overline{h_2\overline{h_1xh_1'}h_2'})
    \\=&|H^F\cap Q^F|^{-2}\sum _{\substack{h,h'\in H^F\\hxh'\in P^F}}f(\overline{hxh'})
\end{split}
\end{equation*}
since $h_2\overline{h_1xh_1'}h_2'=\overline{h_2h_1xh_1'h_2'}$ lies in $(Q\cap L)^F$ if and only if $h_2h_1xh_1'h_2'$ lies in $Q^F$ and the sum over $h_2,h_2'\in H_L^F$ contributes a factor of $|H_L^F|^2$ leaving 
\begin{equation*}
    |H^F\cap U_P^F|^{-2}|H^F\cap M^F|^{-2}|H^F\cap (U_Q\cap L)^F|^{-2}=|H^F\cap Q^F|^{-2}
\end{equation*}
where $U_P$ and $U_Q$ denote the unipotent radicals of $P$ and $Q$ respectively.
\end{proof}

\subsection{Bi-invariant parabolic restriction}
Similar to the definition of induction, define a function
\begin{equation*}
    \Res_{L\subseteq P}^{G/H}:\mathbf{C}[H^F\backslash G^F/H^F]\to \mathbf{C}[H_L^F\backslash L^F/H_L^F]
\end{equation*}
with the formula
\begin{equation*}
    \Res^{G/H}_{L\subseteq P}(f)(x):=\sum _{p\in P^F, \overline{p}=x}f(p).
\end{equation*}
This function is $H_L^F$ bi-invariant. Note that this definition coincides with the usual definition of parabolic restriction of class functions. The first observation is that this operation is the adjoint of parabolic induction.
\begin{proposition}
\label{prop: Ind Res adjoint}
If $f$ is $H^F$ bi-invariant on $G^F$ and $g$ is $H_L^F$ bi-invariant on $L^F$, then
\begin{equation*}
    \langle f,\Ind_{L\subseteq P}^{G/H} g\rangle =|H^F|^2|H^F\cap P^F|^{-2} \langle \Res_{L\subseteq P}^{G/H} f,g\rangle.
\end{equation*}
\end{proposition}
\begin{proof}
Compute
\begin{equation*}
\begin{split}
    \langle f,\Ind_{L\subseteq P}^{G/H}(g)\rangle&=\sum _{x\in G^F}f(x)\overline{\Ind_{L\subseteq P}^{G/H}(g)(x)}
    \\&=|H^F\cap P^F|^{-2}\sum _{x\in G^F}\sum _{\substack{h,h'\in H^F\\ hxh'\in P^F}}f(hxh')\overline{g(\overline{hxh'})}
    \\&=|H^F|^2|H^F\cap P^F|^{-2}\sum _{p\in P^F}f(p)\overline{g(\overline{p})}
\end{split}
\end{equation*}
where the fact that $f$ is $H^F$ bi-invariant is used. On the other hand,
\begin{equation*}
\begin{split}
    \langle \Res_{L\subseteq P}^{G/H}(f),g\rangle&=\sum _{x\in L^F}\sum _{p\in P^F, \overline{p}=x}f(p)\overline{g(x)}
    \\&=\sum _{p\in P^F}f(p)\overline{g(\overline{p})}.
\end{split}
\end{equation*}
\end{proof}

\subsection{Induction of basic functions}
The following results concern only $G=\GL_{2n}$ with $\iota(x)=-J(x^{-1})^TJ$.

Let $L_0$ denote a rational Levi subgroup of the form $\GL_n\times \GL_n$ in $\GL_{2n}$ such that $L_0\cap H\cong \GL_n$ and $L_0/(L_0\cap H)\cong \GL_n$. Let $P_0$ denote a rational $\iota$-stable parabolic for $L_0$. With the conventions taken, an example is given by having $L_0$ act on the first and last $n$ coordinates separately.

Then bi-invariant functions on $L_0$ can be identified with class functions on $\GL_n$, $\iota$-stable Levi and parabolic subgroups can be identified with Levi and parabolic subgroups of $\GL_n$ and bi-invariant parabolic induction corresponds to the usual parabolic induction. 

To see this, note that the map $G\to G^{-\iota}$ restricts to $L_0\to L_0^{-\iota}$ and $L_0^{-\iota}\cong \GL_n$. Then if $f$ is a bi-invariant function on $L^F_0$, there is a unique $H^F_{L_0}$-invariant function $\widetilde{f}$ on $L_0^{-\iota}$ such that $f(x)=\widetilde{f}(x\iota(x)^{-1})$. Also, $H_{L_0}$-conjugacy in $L_0^{-\iota}$ corresponds to $\GL_n$-conjugacy and $\iota$-stable Levi and parabolic subgroups are mapped to Levi and parabolic subgroups of $\GL_n$ by $L\mapsto L^{-\iota}$ and $P\mapsto P^{-\iota}$. Under this correspondence,
\begin{equation*}
\begin{split}
    \Ind_{L\subseteq P}^{L_0/H_{L_0}}(f)(x)&=|\GL_n^F|^{-1}\sum _{h\in \GL_n^F,hxh^{-1}\in P^{-\iota}}\widetilde{f}(\overline{hx\iota(x)^{-1}h^{-1}})
    \\&=\Ind^{G^{-\iota}}_{L^{-\iota}\subseteq P^{-\iota}}(\widetilde{f})(x\iota(x)^{-1})
\end{split}
\end{equation*}
which is just the value of the usual parabolic induction of $f$ from $L^{-\iota}$ to $L_0^{-\iota}$ at the point $x\iota(x)^{-1}$. Thus, for all intents and purposes working in $L_0$ is equivalent to the group situation for $\GL_n$.

For any $\iota$-stable rational Levi subgroup $L\subseteq G$, with $\iota$-stable rational parabolic, 
\begin{equation*}
    L^F\cong \GLq{2n_0}\times \prod \GLq{n_i}\times \GLq{n_i}
\end{equation*}
where $\sum n_i=n$ and $L/H_L\cong \GL_{2n_0}/\Sp_{2n_0}\times \prod \GL_{n_i}$. Let $T=\prod T_i$ be a rational maximal torus of $\prod \GL_{n_i}$. Then if $\theta=\prod \theta_i$ is an irreducible character of $T^F$, let $\zeta_{T}^L(\cdot|\theta)$ be the function on $L^F$ defined by
\begin{equation*}
    \zeta_{T}^{L/H_L}(l|\theta)=\zeta_{T_0}^{\GL_{2n_0}/\Sp_{2n_0}}(l_0|\theta_0)_{q\mapsto q^2}\prod \zeta_{T_i}^{\GL_{n_i}}(l_i|\theta_i)
\end{equation*}
where $l=(l_0,l_1,\cdots)$ and $\zeta_{T_i}^{\GL_{n_i}}(\cdot|\theta_i)$ is viewed as a function on $\GLq{n_i}\times \GLq{n_i}$ under the correspondence described above. These are called \emph{basic functions} and form a basis for the bi-invariant functions on $L^F$. This is a simple extension of the definitions already given for basic functions on symmetric spaces of the form $\GL_{2n}/\Sp_{2n}$ and $\GL_n$ to products of these spaces.

The following proposition is proved in \cite[Theorem 5.3.2]{BKS90} for a particular choice of $L_0$ and $P_0$ but this suffices as all such choices are conjugate under $H^F$.
\begin{proposition}
\label{prop: ind for L_0}
Let $L_0\subseteq P_0$ be defined as above. Then if $T$ is a rational maximal torus of $\GL_n$ and $\theta$ an irreducible character on $T^F$,
\begin{equation*}
    \Ind_{L_0\subseteq P_0}^{G/H}(\zeta_{T}^{\GL_n}(\cdot|\theta))=\frac{|T^F|_{q\mapsto q^2}}{|T^F|}\zeta_{T}^{G/H}(\cdot|\theta).
\end{equation*}
\end{proposition}

The next proposition extends the previous one to any maximal $\iota$-stable Levi and parabolic subgroup.

\begin{proposition}
\label{prop:ind for functions}
Let $L$ be a rational $\iota$-stable Levi subgroup of the form $\GL_{2n}\times \GL_m\times \GL_m$ with $H_L\cong \Sp_{2n}\times \GL_m$ and $P$ a rational $\iota$-stable parabolic with $L$ as its Levi factor. Let $T=T_1\times T_2$ be a rational maximal torus of $\GL_{n+m}$ with $T_1$ and $T_2$ maximal tori in $\GL_n$ and $\GL_m$ respectively, and $\theta$ an irreducible character of $T^F$. Then 
\begin{equation*}
    \Ind_{L\subseteq P}^{G/H}(\zeta_{T}^{L/H_L}(\cdot|\theta))=\frac{|T_2^F|_{q\mapsto q^2}}{|T_2^F|}\zeta_{T}^{G/H}(\cdot|\theta).
\end{equation*}
\end{proposition}

\begin{proof}
Let $L_0$ and $P_0$ be chosen so that there is a parabolic subgroup $P'$ with Levi factor $L_0\cap L$ contained in both $P$ and $P_0$. This can be done by taking a rational $\iota$-stable Borel in $P$, and choosing $P_0$ to be the parabolic corresponding to the subset of all simple roots except the middle one.

Using Proposition \ref{prop: ind for L_0} on the $\GL_{2n}$ factor write
\begin{equation*}
    \Ind_{L\subseteq P}^{G/H}(\zeta_{T}^{L/H_L}(\cdot|\theta))=\frac{|T_1^F|}{|T_1^F|_{q\mapsto q^2}}\Ind_{L\subseteq P}^{G/H}(\Ind_{L\cap L_0\subseteq P'\cap L}^{L/H_L}(\zeta_{T}^{L\cap L_0/H_{L\cap L_0}}(\cdot|\theta))).
\end{equation*}
Then by Proposition \ref{prop: transitivity for functions}, as $P'$ is contained in both $P$ and $P_0$,
\begin{equation*}
    \Ind_{L\subseteq P}^{G/H}\circ\Ind_{L\cap L_0\subseteq P'\cap L}^L=\Ind^{G/H}_{L\cap L_0\subseteq P'}=\Ind_{L_0\subseteq P_0}^{G/H}\circ \Ind_{L\cap L_0\subseteq P'\cap L_0}^{L_0/H_{L_0}}.
\end{equation*}
Now $\Ind_{L\cap L_0\subseteq P'\cap L_0}^{L_0/H_{L_0}}$ is the usual parabolic induction on $\GL_n$ and so
\begin{equation*}
    \Ind_{L\cap L_0\subseteq P'\cap L_0}^{L_0/H_{L_0}}(\zeta_{T}^{L_0\cap L/H_{L_0\cap L}}(\cdot|\theta))=\zeta_{T}^{L_0/H_{L_0}}(\cdot|\theta)
\end{equation*}
and a final application of Proposition \ref{prop: ind for L_0} establishes the result.
\end{proof}

\section{The Characteristic Map}

\label{sec:characteristic map}
In this section, a characteristic map 
\begin{equation*}
\ch:\mathbf{C}[\Spq{2n}\backslash \GLq{2n}/\Spq{2n}]\rightarrow \bigotimes_{f\in O(\mathbf{M})}\Lambda
\end{equation*}
is constructed from the $\Spq{2n}$ bi-invariant functions on $\GLq{2n}$. Here the tensor product $\otimes \Lambda$ is taken over irreducible $f$ and each factor is isomorphic to a copy of the symmetric functions. The notational convention will be to drop the tensor and if $p\in \Lambda$ corresponds to the factor indexed by $f$, then write $p(f)=p(x_{1,f},\dotsc)$. Before proceeding, the theory for $\GLq{n}$ is reviewed because the construction for $\GLq{2n}/\Spq{2n}$ is both very similar and some facts about the characteristic map for $\GLq{n}$ are used in the proofs. The exposition follows \cite{M95}.

\subsection{The $\GLq{n}$ Theory}
Define
\begin{equation*}
\ch_{\GL}: \bigoplus_{n} \mathbf{C}[\GLq{n}]^{\GLq{n}}\rightarrow \bigotimes_{f\in O(\mathbf{M})}\Lambda
\end{equation*}
by
\begin{equation*}
\ch_{\GL_n}(I_{C_\mu})=\prod _{f\in O(\mathbf{M})}q_f^{-n(\mu(f))}P_{\mu(f)}(f;q_f^{-1}),
\end{equation*}
where $I_{C_\mu}$ denotes the indicator function for $C_\mu$ with $\mu:O(\mathbf{M})\rightarrow \mathcal{P}$ and $P_\mu(x;t)$ is the Hall-Littlewood symmetric function.

Define a multiplication on $\oplus \mathbf{C}[\GLq{n}]^{\GLq{n}}$ given by parabolic induction. That is, given two class functions, $f$ on $\GLq{n}$ and $g$ on $\GLq{m}$, define a class function on $\GLq{n+m}$ by embedding $\GLq{n}\times \GLq{m}$ as a standard Levi subgroup, viewing $f\times g$ as a function on this Levi subgroup, extending to a rational parabolic subgroup by the canonical quotient, and then inducing to the full group. With this multiplication, $\ch$ is an isomorphism of graded algebras on $\oplus_n \mathbf{C}[\GLq{n}]^{\GLq{n}}$. In addition, this map is an isometry when $\mathbf{C}[\GLq{n}]^{\GLq{n}}$ is equipped with the usual inner product and $\Lambda$ with the inner product defined by
\begin{equation*}
\langle p_\mu,p_\mu\rangle_{\GL_n}= \prod _{f\in O(\mathbf{M})}z_{\mu(f)}\prod _i\frac{1}{q_f^{\mu(f)_i}-1}
\end{equation*}
for $\mu:O(\mathbf{M})\rightarrow \mathcal{P}$.

Then define
\begin{align*}
\widetilde{p}_n(x)&:=\begin{cases}
p_{n/d(f_x)}(f_x)&\text{if }d(f_x)|n
\\0&\text{else}
\end{cases},
\\\widetilde{p}_n(\xi)&:=\begin{cases}
(-1)^{n-1}\sum _{x\in \mathbf{M}_n}\xi(x)\widetilde{p}_n(x)&\text{if }\xi\in \mathbf{L}_n
\\0&\text{else}
\end{cases},
\\p_n(\varphi)&:=\widetilde{p}_{nd(\varphi)}(\xi),
\end{align*}
where $\xi\in \varphi$ for $\varphi\in O(\mathbf{L})$. Since they are algebraically independent, the $p_n(\varphi)$ may be viewed as power sum symmetric functions in "dual variables $x_{i,\varphi}$" (that is, they may be formally viewed as symmetric functions in these variables, and the definition of $p_n(\varphi)=p_n(x_{i,\varphi})$ already available can be used to define other symmetric functions, e.g. $s_\lambda(\varphi), e_\lambda(\varphi)$).

As a matter of convention, $\mu$ will denote functions $O(\mathbf{M})\rightarrow \mathcal{P}$ and $\lambda$ will denote functions $O(\mathbf{L})\rightarrow \mathcal{P}$, and symmetric functions indexed by $\mu$ will always be in variables $x_{i,f}$ for $f\in O(\mathbf{M})$ and symmetric functions indexed by $\lambda$ will always be in the variables $x_{i,\varphi}$. Symmetric functions labeled by a partition valued function are interpreted as a product. Thus,
\begin{equation*}
p_\mu=\prod _{f\in O(\mathbf{M})}p_{\mu(f)}(f).
\end{equation*}

Now the characteristic map on the characters can be computed as
\begin{equation*}
\ch_{\GL_n}(\chi_\lambda)=s_{\lambda}
\end{equation*}
where the $s_\lambda$ denote the Schur functions. Note by orthonormality of the characters that the inner product on the dual variables is just the standard Hall inner product. That is,
\begin{equation*}
\langle p_\lambda,p_\lambda\rangle_{\GL_n}=z_\lambda.
\end{equation*}

Finally, note that if $T$ is a rational maximal torus in $\GL_n$ with $T^F\cong \mathbf{M}_{k_1}\times \dotsm\times \mathbf{M}_{k_r}$ and $\theta$ an irreducible character of $T^F$, then
\begin{equation*}
\ch_{\GL_n}(\zeta_{T}^{\GL_n}(\cdot|\theta))=(-1)^{n-r}\prod _i p_{k_i/d(\varphi_i)}(\varphi_i)=(-1)^{n-l(\lambda)}p_\lambda,
\end{equation*}
where $\varphi_i$ is the orbit of $\theta|_{\mathbf{M}_{k_i}}$ and $\lambda$ is the combinatorial data associated to $(T,\theta)$. In fact, this could be used as a definition of the Deligne-Lusztig characters for $\GLq{n}$.

\subsection{Definition of the Characteristic Map}
Now return to the case of interest. Define the inner product $\langle f,g\rangle =\sum _{x\in \GLq{2n}}f(x)\overline{g(x)}$ on $\mathbf{C}[\Spq{2n}\backslash \GLq{2n}/\Spq{2n}]$ and the inner product on $\otimes \Lambda$ by 
\begin{equation*}
\langle p_\mu,p_\mu\rangle=z_\mu \prod _{f\in O(\mathbf{M})}\prod _i\frac{1}{q_f^{2\mu(f)_i}-1}.
\end{equation*}

Define the characteristic map $\ch$ by
\begin{equation*}
\ch(I_{H^Fg_\mu H^F})=\prod _{f\in O(\mathbf{M})}q_f^{-2n(\mu(f))}P_{\mu(f)}(f;q_f^{-2})
\end{equation*}
and extending linearly.

Note that the coefficients of $p_\mu$ in $\ch(I_{H^Fg_\mu H^F})$ (which are rational in $q$), are the same as those in $\ch_{\GL_n}(I_{C_\mu})$ but with $q$ replaced with $q^2$. It is natural to expect in light of Theorem \ref{thm: spherical function formula} that the same is true for the spherical functions. To establish this, the following lemma is needed.

\begin{lemma}
Let $t$ denote a formal variable and use $\zeta_T^{\GL_n}(C_\mu|\theta,t)$ to denote the Deligne-Lusztig character $\zeta_T^{\GL_n}(\cdot|\theta)$, which is a rational function in $q$, in terms of $t$. Here, $T^F\cong \prod_{i=1}^r \mathbf{M}_{k_i}$ and $\theta=\prod \theta_i$ with $\theta_i\in \varphi_i$. Then
\begin{equation*}
\begin{split}
\sum _{\mu}\zeta_T^{\GL_n}(M_\mu|\theta,t)\prod _{f\in O(\mathbf{M})}t_f^{-n(\mu(f))}P_{\mu(f)}(f;t_f^{-1})&= (-1)^{n-r}\prod _{i}p_{k_i/d(\varphi_i)}(\varphi_i)
\\&=(-1)^{n-l(\lambda)}p_\lambda
\end{split}
\end{equation*}
as polynomials with coefficients in $\mathbf{C}(t)$, where $\lambda$ is the combinatorial data associated to $(T,\theta)$ and $M_\mu$ lies in the conjugacy class indexed by $\mu$.
\end{lemma}
\begin{proof}
This proof follows that of Theorem 4.2 in \cite{TV07}, where it is proven for the case of $t=-q$, but the proof given works for a formal parameter.

It suffices to show that that the coefficient of $P_{\mu(f)}(f;q_f^{-1})$ in $\prod p_{k_i/d(\varphi_i)}(\varphi_i)$ is $\zeta_T^{\GL_n}(C_\mu|\theta,t)$. Write
\begin{equation*}
\begin{split}
(-1)^{n-r}\prod_i p_{k_i/d(\varphi_i)}(\varphi_i)&=\prod_i \sum _{x_i\in \mathbf{M}_{k_i}}\theta_i(x_i)p_{k_i/d(f_{x_i})}(f_{x_i})
\\&=\sum _{x\in T}\theta(x)\prod _i p_{k_i/d(f_{x_i})}(f_{x_i}).
\end{split}
\end{equation*}
Now rewrite $\prod _i p_{k_i/d(f_{x_i})}(f_{x_i})=\prod _{f\in O(\mathbf{M})}p_{\gamma_x(f)}(f)$ and use the fact that the Green polynomials are the change of basis from power sum to Hall-Littlewood polynomials to obtain
\begin{equation*}
\begin{split}
&(-1)^{n-r}\prod p_{k_i/d(\varphi_i)}(\varphi_i)
\\=&\sum _{x\in T^F}\theta(x)\prod _{f\in O(\mathbf{M})}\left(\sum _{|\mu(f)|=|\gamma_x(f)|}Q_{\gamma_x(f)}^{\mu(f)}(t_f)t_f^{-n(\mu(f))}P_{\mu(f)}(f,t_f^{-1})\right)
\\=&\sum _{\mu}\left(\sum _{x\in T,|\gamma_x(f)|=|\mu(f)|}\theta(t)Q_{\gamma_x}^{\mu}(t)\right)\prod _{f\in O(\mathbf{M})}t_f^{-n(\mu(f))}P_{\mu(f)}(f,t_f^{-1})
\\=&\sum _\mu \zeta_T^{\GL_n}(C_\mu|\theta,t)\prod _{f\in O(\mathbf{M})}t_f^{-n(\mu(f))}P_{\mu(f)}(f,t_f^{-1})
\end{split}
\end{equation*}
as required.
\end{proof}

\begin{proposition}
\label{prop: characteristic of DL character}
Let $T$ be a rational maximal torus of $\GL_n$ with $T^F\cong \mathbf{M}_{k_1}\times \dotsm\times \mathbf{M}_{k_r}$ and $\theta$ a character of $T^F$. Then
\begin{equation*}
\ch(\zeta_{T}^{\GL_{2n}/\Sp_{2n}}(\cdot|\theta))=(-1)^{n-r}\prod _{i}p_{k_i/d(\varphi_i)}(\varphi_i)=(-1)^{n-l(\lambda)}p_\lambda,
\end{equation*}
where $\varphi_i$ is the $F$-orbit of $\theta|_{\mathbf{M}_{k_i}}\in \mathbf{L}$ and $\lambda$ is the combinatorial data associated to $(T,\theta)$. 
\end{proposition}

\begin{proof}
Write
\begin{equation*}
\zeta_{T}^{\GL_{2n}/\Sp_{2n}}(\cdot|\theta)=\sum_\mu \zeta_{T}^{\GL_n}(M_\mu |\theta)_{q\mapsto q^2} I_{H^Fg_\mu H^F}
\end{equation*}
and apply the characteristic function, giving
\begin{equation*}
\sum _{\mu}\zeta_{T}^{\GL_n}(M_\mu|\theta)_{q\mapsto q^2}\prod _{f\in O(\mathbf{M})}q_f^{-2n(\mu(f))}P_{\mu(f)}(f;q_f^{-2})=(-1)^{n-r}\prod_i p_{k_i/d(\varphi_i)}(\varphi_i)
\end{equation*}
by the lemma with $t=q^2$.
\end{proof}

Now $\ch(\phi_\lambda)$ may be computed.
\begin{proposition}
\label{prop: ch of spherical func}
Let $\lambda:O(\mathbf{L})\rightarrow \mathcal{P}$ with $\|\lambda\|=n$. Then
\begin{equation*}
\ch(\phi_{\lambda})=\frac{(-1)^{|\lambda|}}{\psi_n(q^2)}\prod _{\varphi\in O(\mathbf{L})}q_\varphi^{-n(\lambda'(\varphi))} J_{\lambda(\varphi)}(q_\varphi,q_\varphi^2).
\end{equation*}
\end{proposition}

\begin{proof}
First, note that
\begin{equation*}
\begin{split}
\phi_\lambda=&\frac{(-1)^{|\lambda|}}{|W(\lambda)|}\sum _{w\in W(\lambda)}\left(\prod _{\varphi\in O(\mathbf{L})}q_\varphi^{-n(\lambda(\varphi)')}c_{\lambda(\varphi)'}(q_\varphi^2,q_\varphi)d_{\lambda(\varphi)'}(w)(q_\varphi)\right)
\\&\qquad\qquad\qquad\qquad\times\sgn(w)\frac{\zeta_{T_w}^{\GL_n}(\cdot|\theta_w)_{q\mapsto q^2}}{|T^F_w|\zeta_{T_w}^{\GL_n}(1|\theta_w)_{q\mapsto q^2}}
\end{split}
\end{equation*}
from Theorem \ref{thm: spherical function formula}. Recall that $c_\lambda$ denotes the scaling from $P_\lambda$ to $J_\lambda$, $d_\lambda$ denotes the change of basis from power sum symmetric functions to Macdonald symmetric functions and $\zeta_T^G$ denotes a Deligne-Lusztig character.

Now $\zeta_{T_w}^{\GL_n}(1)=(-1)^{n-l(w)}|\GL_{n}^F/T_w^F|_{q'}$ (see \cite[Thoerem 7.5.1]{C93}), where for an integer $n$, $n_{q'}$ denotes the $q$-free part. Here $l(w)$ denotes the length of the partition corresponding to $w$ and not the Coxeter length. Thus, $\zeta_{T}^{\GL_n}(1)=(-1)^{n-l(w)}\psi_n(q^2)/|T^F_w|_{q\mapsto q^2}$. Then apply $\ch$, obtaining
\begin{equation*}
\begin{split}
&\frac{(-1)^{|\lambda|}}{|W(\lambda)|}\sum _{w\in W(\lambda)}\left(\prod _{\varphi\in O(\mathbf{L})}q_\varphi^{-n(\lambda(\varphi)')}c_{\lambda(\varphi)'}(q_\varphi^2,q_\varphi)d_{\lambda(\varphi)'}(w)(q_\varphi)\right)
\\&\qquad\qquad\qquad\qquad\times(-1)^{|\lambda|-l(w)}\frac{\prod_{\varphi\in O(\mathbf{L})}p_{w(\varphi)}(\varphi) |T^F_w|_{q\mapsto q^2}}{|T^F_w|\psi_n(q^2)}
\end{split}
\end{equation*}
as applying Proposition \ref{prop: characteristic of DL character} gives
\begin{equation*}
\ch(\zeta_{T_w}^{\GL_n}(\cdot|\theta_{w})_{q\mapsto q^2})=(-1)^{n-l(w)}\prod _{\varphi\in O(\mathbf{L})}p_{w(\varphi)}(\varphi).
\end{equation*}

Next, notice that if $\omega_{q,t}$ denotes the automorphism of $\otimes \Lambda$ which sends $p_n(\varphi)$ to $(-1)^{n-1}(q_\varphi^{n}-1)/(t_\varphi^{n}-1)p_n(\varphi)$ then
\begin{equation*}
\omega_{q^2,q}\prod_{\varphi\in O(\mathbf{L})}p_{w(\varphi)}=(-1)^{|\lambda|-l(w)}\frac{\prod_{\varphi\in O(\mathbf{L})}p_{w(\varphi)}(\varphi) |T^F_w|_{q\mapsto q^2}}{|T^F_w|}
\end{equation*}
and so
\begin{equation*}
\begin{split}
\ch(\phi_\lambda)&=\frac{(-1)^{|\lambda|}}{\psi_n(q^2)}\prod _{\varphi\in O(\mathbf{L})}\left({q_\varphi^{-n(\lambda(\varphi)')}c_{\lambda(\varphi)'}(q_\varphi^2,q_\varphi)}\sum _{w\in S_{|\lambda(\varphi)|}}d_{\lambda(\varphi)'}(w)(q_\varphi)\omega_{q^2,q}p_w(\varphi)\right)
\\&=\omega_{q^2,q}\frac{(-1)^{|\lambda|}}{\psi_n(q^2)}\prod _{\varphi\in O(\mathbf{L})}\left(q_\varphi^{-n(\lambda(\varphi)')}J_{\lambda(\varphi)'}(\varphi;q_\varphi^{2},q_\varphi)\right)
\\&=\frac{(-1)^{|\lambda|}}{\psi_n(q^2)}\prod _{\varphi\in O(\mathbf{L})}q_\varphi^{-n(\lambda(\varphi)')} J_{\lambda(\varphi)}(\varphi;q_\varphi,q_\varphi^2)
\end{split}
\end{equation*}
\end{proof}

\subsection{The Isometry Property}
Next, it will be shown that the map $\ch$ is an isometry, up to a constant.

\begin{lemma}
\label{lemma: isometry}
The map $\ch$ satisfies
\begin{equation*}
\langle \phi,\psi\rangle=q^{-n}|H^F|^2\langle \ch(\phi),\ch(\psi)\rangle.
\end{equation*}
\end{lemma}
\begin{proof}
It suffices to compute the norms of indicator functions. It is clear that $\langle I_{H^Fg_\mu H^F},I_{H^Fg_\mu H^F}\rangle=|H^Fg_\mu H^F|$. Now
\begin{equation*}
\langle \ch(I_{H^Fg_\mu H^F}),\ch(I_{H^Fg_\mu H^F})\rangle =(\langle \ch(I_{C_\mu}),\ch(I_{C_\mu})\rangle _{\GL_n})_{q\mapsto q^2},
\end{equation*}
where $C_\mu$ denotes the conjugacy class of $\GLq{n}$ associated to $\mu$. This can be shown by noting that $\ch(I_{H^Fg_\mu H^F})=\ch(I_{C_\mu})_{q\mapsto q^2}$, and expanding into power sum symmetric functions, and then using the fact that $\langle p_\mu,p_\mu\rangle=(\langle p_\mu,p_\mu\rangle _{\GL_n})_{q\mapsto q^2}$. But 
\begin{equation*}
(\langle \ch(I_{C_\mu}),\ch(I_{C_\mu})\rangle _{\GL_n})_{q\mapsto q^2}=\frac{1}{a_\mu(q^2)}
\end{equation*}
because the characteristic map for $\GLq{n}$ is an isometry. Finally,
\begin{equation*}
\frac{|H^F|^2q^{-n}}{a_\mu(q^2)}=|H^F|\frac{|\GLq{n}|_{q\mapsto q^2}}{a_\mu(q^2)}
\end{equation*}
because $|\Spq{2n}|=q^{n^2}\prod (q^{2i}-1)$ and $|\GLq{n}|_{q\mapsto q^2}=q^{n^2-n}\prod (q^{2i}-1)$. But then 
\begin{equation*}
|C_\mu|_{q\mapsto q^2}=\frac{|\GLq{n}|_{q\mapsto q^2}}{a_\mu(q^2)}
\end{equation*}
since $a_\mu(q)$ is the size of the centralizer of an element in $C_\mu$, and this gives
\begin{equation*}
q^{-n}|H^F|^2\langle \ch(I_{H^Fg_\mu H^F}),\ch(I_{H^Fg_\mu H^F})\rangle=|H^Fg_\mu H^F|
\end{equation*}
using Proposition \ref{prop: double coset size}. This shows that $\ch$ is an isometry up to the specified constant.
\end{proof}

\begin{remark}
The inner product on $\mathbf{C}[\Spq{2n}\backslash \GLq{2n}/\Spq{2n}]$ could be renormalized so that $\ch$ is actually an isometry. This would also remove some of the constants appearing in Proposition \ref{prop: Ind Res adjoint} and so in some sense this inner product would be more natural although to avoid confusion with the usual one this is not done. These constants also appear in the characteristic map for the Gelfand pair $S_{2n}/B_n$, see \cite[VII, \S 2]{M95}.
\end{remark}

\begin{remark}
Lemma \ref{lemma: isometry} could also be proven by computing the norms of spherical functions. Strictly speaking, this is not necessary, but the computation helps illustrate the use of the dual variables so it is included. This computation naturally leads to a proof of Theorem \ref{thm: spherical function formula}, with the triangularity between the $P_\lambda$ and $m_\mu$ the remaining condition to check. This is basically how Theorem \ref{thm: spherical function formula} is proven in \cite{BKS90} and so for the sake of brevity is not reproduced here.
\end{remark}

To compute the analogue for the spherical functions, first $\langle,\rangle$ must be computed on the dual variables $x_{i,\varphi}$ (since the indicator functions are in terms of $f\in O(\mathbf{M})$ and the spherical functions are in terms of $\varphi\in O(\mathbf{L})$).

\begin{proposition}
For $\lambda:O(\mathbf{L})\rightarrow \mathcal{P}$,
\begin{equation*}
\langle p_\lambda,p_\lambda\rangle =z_\lambda\prod _{\varphi\in O(\mathbf{L})}\prod_i \frac{q_\varphi^{\lambda(\varphi)_i}-1}{q_\varphi^{2\lambda(\varphi)_i}-1}.
\end{equation*}
\end{proposition}

\begin{proof}
Let $\overline{p_\lambda(\varphi)}$ denote the symmetric function obtained by expanding $p_\lambda(\varphi)$ in terms of $p_\lambda(f)$ and taking the complex conjugate of all coefficients, extended multiplicatively. First note that
\begin{equation*}
\begin{split}
\sum _{\varphi\in \mathbf{L}_n}\widetilde{p}_n(\varphi)\otimes \overline{\widetilde{p}_n(\varphi)}&=\sum _{\varphi\in \mathbf{L}_n}\sum _{x,y\in \mathbf{M}_n}\varphi(x)\overline{\varphi}(y)\widetilde{p}_n(x)\otimes \widetilde{p}_n(y)
\\&=(q^n-1)\sum _{x\in \mathbf{M}_n}\widetilde{p}_n(x)\otimes \widetilde{p}_n(y).
\end{split}
\end{equation*}
Then
\begin{equation*}
\begin{split}
\sum _{\varphi\in \mathbf{L}_n}\widetilde{p}_n(\varphi)\otimes \overline{\widetilde{p}_n(\varphi)}&=\sum _{\varphi\in O(\mathbf{L}),d(\varphi)|n}d(\varphi)p_{n/d(\varphi)}(\varphi)\otimes \overline{p_{n/d(\varphi)}(\varphi)}
\end{split}
\end{equation*}
and
\begin{equation*}
\sum _{x\in \mathbf{M}_n}\widetilde{p}_n(x)\otimes \widetilde{p}_n(x)=\sum _{f\in O(\mathbf{L}),d(f)|n}d(f_x)p_{n/d(f_x)}(f_x)\otimes p_{n/d(f_x)}(f_x)
\end{equation*}
and so multiplying by $\frac{q^{2n}-1}{n(q^n-1)}$ and summing over all $n$ gives
\begin{equation*}
\sum _{n\geq 1}\frac{1}{n}\sum _{f\in O(\mathbf{M})}(q_f^{2n}-1) p_n(f)\otimes p_n(f)=\sum _{n\geq 1}\frac{1}{n}\sum _{\varphi\in O(\mathbf{L})}\frac{q_\varphi^{2n}-1}{q_\varphi^n-1} p_n(\varphi)\otimes \overline{p_n(\varphi)}.
\end{equation*}

Finally, exponentiate both sides to obtain
\begin{equation*}
\sum _{\lambda}\frac{1}{z_\lambda}\left(\prod_{\varphi\in O(\mathbf{L})}\prod_i \frac{q_\varphi^{2\lambda(\varphi)_i}-1}{q_\varphi^{\lambda(\varphi)_i}-1}\right)p_\lambda\otimes \overline{p_\lambda} =\sum _{\mu}\frac{1}{z_\mu}\left(\prod _{f\in O(\mathbf{M})}\prod _{i}(q_f^{2\mu(f)_i}-1)\right)p_\mu\otimes p_\mu
\end{equation*}
and this power series identity implies that 
\begin{equation*}
\langle p_\lambda,p_\lambda\rangle =z_\lambda\prod_{\varphi\in O(\mathbf{L})}\prod _i \frac{q_\varphi^{\lambda(\varphi)_i}-1}{q_\varphi^{2\lambda(\varphi)_i}-1}.
\end{equation*}
\end{proof}

Now the norms of spherical functions may be computed as follows. 
\begin{proof}[Alternative proof of Lemma \ref{lemma: isometry}]
Note that for any spherical function, it is always the case that
\begin{equation*}
\langle \phi_\lambda,\phi_\lambda\rangle =\frac{|\GLq{2n}|}{d_{\lambda\cup\lambda}}
\end{equation*}
(see e.g. \cite[VII, \S 1]{M95}). Now compute
\begin{equation*}
\begin{split}
\frac{|H^F|^2}{q^n}\langle \ch(\phi_{\lambda}),\ch(\phi_{\lambda})\rangle =\frac{|H^F|^2}{q^{n}\psi_n(q^2)^2}\prod _{\varphi\in O(\mathbf{L})}q_\varphi^{-2n(\lambda(\varphi)')}\langle J_{\lambda}(q,q^2),J_{\lambda}(q,q^2)\rangle.
\end{split}
\end{equation*}
Because $\langle J_\lambda(q,q^2),J_\lambda(q,q^2)\rangle =c_\lambda(q,q^2) c_{\lambda'}(q,q^2)=H_{\lambda\cup\lambda}(q)$, this is equal to
\begin{equation*}
|\GLq{2n}|\psi_{2n}(q)^{-1}\prod _{\varphi\in O(\mathbf{L})}q_\varphi^{-n((\lambda(\varphi)\cup\lambda(\varphi))')}H_{\lambda\cup\lambda}(q)
\end{equation*}
and finally the dimension formula
\begin{equation*}
    d_{\lambda\cup\lambda}=\psi_{2n}(q)\prod _{\varphi\in O(\mathbf{L})}q_\varphi^{n((\lambda(\varphi)\cup\lambda(\varphi))')}H_{\lambda\cup \lambda}^{-1}(q)
\end{equation*}
gives
\begin{equation*}
\frac{|H^F|^2}{q^n}\langle \ch(\phi_{\lambda}),\ch(\phi_{\lambda})\rangle =\frac{|\GLq{2n}|}{d_{\lambda\cup \lambda}}
\end{equation*}
as desired.
\end{proof}

\subsection{Induction}
Let $G=\GL_{2(n+m)}$ and let $H=\Sp_{2(n+m)}$. Let $L$ be an $\iota$-stable Levi subgroup with $L^F=\GLq{2n}\times \GLq{m}\times \GLq{m}$, such that $L^F\cap H^F=\Spq{2n}\times \GLq{m}$. Let $P$ be a rational $\iota$-stable parabolic with $L$ as its Levi factor. With the conventions taken, an explicit realization of these subgroups are given by
\begin{equation*}
    \setlength{\arraycolsep}{3pt}
    L=\left\{\left(\begin{array}{cccc}
         *&0&0&0\\
         0&*&*&0\\
         0&*&*&0\\
         0&0&0&*
    \end{array}\right)\right\}, \quad
    P=\left\{\left(\begin{array}{cccc}
         *&*&*&*\\
         0&*&*&*\\
         0&*&*&*\\
         0&0&0&*
    \end{array}\right)\right\}
\end{equation*}
where the sizes of the rows and columns are $m,n,n,m$ and it is easy to check that these are both stable under $\iota$. Then the function
\begin{equation*}
    \Ind_{L\subseteq P}^{G/H}:\mathbf{C}[H_L^F\backslash L^F/H_L^F]\to\mathbf{C}[H^F\backslash G^F/H^F]
\end{equation*}
may be viewed as taking $f$ a $\Spq{2n}$ bi-invariant function on $\GLq{2n}$ and $g$ a class function on $\GLq{m}$ (since there is an isomorphism $\GL_m\times \GL_m/\GL_m\cong \GL_m$) and producing a $\Spq{2(n+m)}$ bi-invariant function on $\GLq{2(n+m)}$, denoted by $f\ast g$. This can be done for any $n,m$ and so defines a graded bilinear map
\begin{equation*}
\begin{split}
    &\bigoplus_n \mathbf{C}[\Spq{2n}\backslash \GLq{2n}/\Spq{2n}]\times \bigoplus_n \mathbf{C}[\GLq{n}]^{\GLq{n}}
    \\&\qquad\qquad\to \bigoplus_n \mathbf{C}[\Spq{2n}\backslash \GLq{2n}/\Spq{2n}].
\end{split}
\end{equation*}

If $f_1,f_2$ are functions on $G_1$ and $G_2$ respectively, let $f_1\times f_2$ denote the function on $G_1\times G_2$ given by $(f_1\times f_2)(x_1,x_2)=f_1(x_1)f_2(x_2)$. Then
\begin{equation*}
\begin{split}
    f\ast g(x):&=\Ind_{L\subseteq P}^{G/H}(f\times g)
    \\&=\frac{1}{|\Sp_{2n}^F|^{2}|\GL_m^F|^{4}q^{m(m+1)+4nm}}\sum _{\substack{h,h'\in \Sp_{2(n+m)}^F\\hxh'\in P^F}}(f\times g)(\overline{hxh'}).
\end{split}
\end{equation*}

\begin{remark}
Zelevinsky showed that $R=\oplus \mathbf{C}[\GLq{n}]^{\GLq{n}}$ can be given the structure of a positive self-adjoint Hopf algebra \cite{Z81}. Then the operation just defined turns
\begin{equation*}
    \bigoplus_n \mathbf{C}[\Spq{2n}\backslash \GLq{2n}/\Spq{2n}]
\end{equation*}
into an $R$-module. Similarly, the restriction operation defined in Section \ref{sec: induction for functions} turns it into an $R$-comodule and the two operations are adjoint. This gives the structure of a positive self-adjoint module over $R$ as defined in \cite{vL91}.
\end{remark}

The key result of this section is the following compatibility with the characteristic map.
\begin{theorem}
\label{thm: multiplicativity}
Let $f$ be an $\Spq{2n}$ bi-invariant function on $\GLq{2n}$ and let $g$ be a $\GLq{m}$ bi-invariant function on $\GLq{m}\times \GLq{m}$ (or equivalently a class function on $\GLq{m}$). Then
\begin{equation*}
    \ch(f\ast g)=\ch(f)\omega\omega_{q^2,q}\ch_{\GL_m}(g).
\end{equation*}
\end{theorem}
\begin{proof}
It suffices to show that if $T_1$ and $T_2$ are rational maximal tori in $\GL_{n}$ and $\GL_m$ respectively, and if $\theta_1$ and $\theta_2$ are characters of $T_1^F$ and $T_2^F$, then
\begin{equation*}
    \zeta_{T_1}^{\GL_{2n}/\Sp_{2n}}(\cdot|\theta_1)\ast \zeta_{T_2}^{\GL_{m}}(\cdot|\theta_2)=\frac{|T_2^F|_{q\mapsto q^2}}{|T_2^F|}\zeta^{G/H}_{T_1\times T_2}(\cdot|\theta_1\times \theta_2).
\end{equation*}
This is because applying the characteristic map implies
\begin{equation*}
    \ch\left(\zeta_{T_1}^{\GL_{2n}/\Sp_{2n}}(\cdot|\theta_1)\ast \frac{|T_2^F|_{q\mapsto q^2}}{|T_2^F|}\zeta_{T_2}^{\GL_{m}}(\cdot|\theta_2)\right)=(-1)^{n+m-l(\lambda)}\frac{|T_2^F|_{q\mapsto q^2}}{|T_2^F|}p_{\lambda}
\end{equation*}
which is exactly the desired result, and as the basic functions form a basis and the equation is linear in $f,g$, the result follows.

Now let $L\subseteq P$ be the Levi and parabolic defining the $\ast$ operation defined above. Then
\begin{equation*}
    \zeta_{T_1}^{\GL_{n}}(\cdot|\theta_1)\ast \zeta_{T_2}^{\GL_{m}}(\cdot|\theta_2)=\Ind_{L\subseteq P}^{G/H}(\zeta_{T_1\times T_2}^{L/H_L}(\cdot|\theta_1\times \theta_2))
\end{equation*}
and the desired equality follows directly from Proposition \ref{prop:ind for functions}.
\end{proof}

\begin{remark}
The above induction operation can be defined for any Levi subgroup with a rational $\iota$-stable parabolic. However, as every such Levi subgroup is of the form $\GL_{2n}\times \prod \GL_{n_i}\times \GL_{n_i}$, there is no real gain to considering more general subgroups because the induction can be broken into two steps, first all the $\GL_{n_i}\times \GL_{n_i}/\GL_{n_i}\cong \GL_{n_i}$ factors, and then the remaining $\GL_{2n_0}/\Sp_{2n_0}$ factor, and the first step simply corresponds to the classical theory for $\GL_n$.

It would be very interesting if an analogous operation could be defined for the Levi subgroups $\GL_{2n}\times \GL_{2m}$ with $\iota$-fixed points $\Sp_{2n}\times \Sp_{2m}$. Unfortunately, it is not hard to see that this cannot be given by the same formula (and so does not correspond to the bi-invariant parabolic induction defined by Grojnowski \cite{G92}). For example, the parabolic induction of the indicator function for a double coset generated by a unipotent element may be non-zero on double cosets generated by a non-unipotent element. 

Such an operation would define a multiplication operation on 
\begin{equation*}
    \bigoplus_n \mathbf{C}[\Spq{2n}\backslash \GLq{2n}/\Spq{2n}]
\end{equation*}
analogous to the ones defined for the more classical examples of a characteristic map. A nice formula for this operation would have consequences for the structure constants of the Macdonald polynomials (the Macdonald analogue of the Littlewood-Richardson coefficients) similar to the applications in \cite{M95} in the $S_{2n}/B_n$ case (or Theorems \ref{thm: positivity} and \ref{thm: vanishing}) on positivity and vanishing.
\end{remark}

\section{Schur expansion of skew Macdonald polynomials}
\label{sec: LR coeff}
In this section, Theorem \ref{thm: positivity} on positivity of the coefficients of the Schur expansion of skew Macdonald polynomials with parameters $(q,q^2)$ is proven. The argument for Theorem \ref{thm: positivity} is in the same spirit as that of Macdonald \cite[VII, \S 2]{M95} for a similar result in the Jack case with parameter $2$. A condition for the vanishing of these coefficients is also given and in the special case of a non-skew Macdonald polynomial, relates to the classical Littlewood-Richardson coefficients.

The strategy will be to use the characteristic map and the mixed product to reinterpret $C_{\lambda/\mu}^\nu(q,q^2)$ in terms of bi-invariant functions on $\GLq{2n}$ and then utilize general facts about Gelfand pairs and the formula for parabolic induction to establish the results.

\subsection{Skew Macdonald polynomials}
If $\lambda/\mu$ is a skew shape, the skew Macdonald polynomial $P_{\lambda/\mu}(x;q,t)$ is defined by requiring
\begin{equation*}
    \langle P_{\lambda/\mu}(q,t),f\rangle=\langle P_\lambda(q,t),Q_\mu(q,t) f\rangle
\end{equation*}
for all symmetric functions $f$. Define the coefficients $C_{\lambda/\mu}^\nu(q,t)$ by
\begin{equation*}
    P_{\lambda/\mu}(x;q,t)=\sum _{\nu}C_{\lambda/\mu}^\nu(q,t)s_\nu(x).
\end{equation*}
These are the coefficients of the Schur expansion of $P_{\lambda/\mu}$.

\begin{remark}
Since 
\begin{equation*}
    P_{\lambda/\mu}(q,t)=\left(\frac{q}{t}\right)^{|\lambda|-|\mu|}P_{\lambda/\mu}(q^{-1},t^{-1}),
\end{equation*}
Theorem \ref{thm: positivity} immediately extends to parameters $(q^{-1},q^{-2})$ for $q$ an odd prime power.
\end{remark}

\subsection{Positive-definite functions}
For any finite group $G$, a \emph{positive-definite} function on $G$ is a function $f:G\to \mathbf{C}$ such that the matrix indexed by $G$ whose $x,y$ entry is $f(x^{-1}y)$ is a positive-definite matrix. A key fact is that any bi-invariant positive-definite function is a positive linear combination of spherical functions.

\begin{proposition}[{\cite[VII, \S 1]{M95}}]
\label{prop: spherical function positivity}
Let $G/H$ be a Gelfand pair. Any spherical function is positive-definite, and moreover if $f$ is an $H$ bi-invariant function on $G$ that is positive-definite, then for any spherical function $\phi$ on $G$, $\langle f,\phi\rangle\geq 0$.
\end{proposition}

If $\alpha:G\to H$ is a group homomorphism, then given functions $f:G\to \mathbf{C}$ and $g:H\to \mathbf{C}$, define the functions $\alpha^*g:G\to \mathbf{C}$, or the pullback, and $\alpha_*f:H\to \mathbf{C}$, or the pushforward, by
\begin{align*}
    \alpha^*g(x)&=g(\alpha(x))\\
    \alpha_*f(x)&=\sum _{\alpha(y)=x}f(y).
\end{align*}
From the definition, it's clear that $(\alpha\circ \beta)^*=\beta^*\circ \alpha^*$ and $(\alpha\circ\beta)_*=\alpha_*\circ\beta_*$. There are two basic properties that are needed, namely that pullback and pushforward are adjoint and that pushforward and pullback preserve positive-definite functions.
\begin{lemma}
\label{lemma: push-pull adjoint}
If $\alpha:G\to H$, and $f:H\to \mathbf{C}$ and $g:G\to \mathbf{C}$, then $\langle \alpha^*f,g\rangle=\langle f,\alpha_*g\rangle$.
\end{lemma}
\begin{proof}
Note that
\begin{equation*}
    \langle \alpha^*f,g\rangle=\sum _{x\in G}f(\alpha(x))g(x)
\end{equation*}
and
\begin{equation*}
    \langle f,\alpha_*g\rangle=\sum _{y\in H}f(y)\sum _{\alpha(x)=y}g(x)
\end{equation*}
which are equal.
\end{proof}
\begin{lemma}
\label{lem: pos-def push-pull}
If $f:G\to \mathbf{C}$ is positive-definite, and $\alpha:G\to H$ and $\beta: H\to G$ are group homomorphisms, then $\alpha_*f$ and $\beta^*f$ are also positive-definite.
\end{lemma}
\begin{proof}
Note positive definiteness is equivalent to having
\begin{equation*}
    \sum _{x,y\in G}f(x^{-1}y)\overline{h(x)}h(y)\geq 0
\end{equation*}
for all functions $h:G\to \mathbf{C}$. Then
\begin{equation*}
\begin{split}
    \sum_{x,y\in H} \beta^*f(x^{-1}y)\overline{h(x)}h(y)&=\sum_{x,y\in G}f(x^{-1}y)\overline{\beta_*h(x)}\beta_*h(y)
    \geq 0
\end{split}
\end{equation*}
so $\beta^*f$ is positive-definite.

For $\alpha_*f$, note that $\alpha$ can always be factored as a surjection and an injection. If $\alpha$ is surjective, then
\begin{equation*}
\begin{split}
    \sum_{x,y\in H} \alpha_*f(x^{-1}y)\overline{h(x)}h(y)&=|\ker \alpha|^{-1}\sum_{x,y\in G} f(x^{-1}y)\overline{\alpha^*h(x)}\alpha^*h(y)\geq 0
\end{split}
\end{equation*}
and if $\alpha$ is injective, then
\begin{equation*}
\begin{split}
    \sum _{x,y\in H}\alpha_*f(x^{-1}y)\overline{h(x)}h(y)&=\sum _{x,y\in H, x^{-1}y\in \alpha(H)}f(x^{-1}y)\overline{h(x)}h(y)
    \\&=\sum _{z\in G/\alpha(H)}\sum _{x,y\in H}f(x^{-1}y)\overline{h(zx)}h(zy)
\end{split}
\end{equation*}
is non-negative as each summand is non-negative.
\end{proof}

Next, it is shown that the parabolic restriction of bi-invariant functions takes positive-definite functions to positive-definite functions.
\begin{lemma}
Let $f$ be a positive-definite $\Spq{2(n+m)}$ bi-invariant function on $\GLq{2(n+m)}$, and let $L\subseteq P$ be any rational $\iota$-stable parabolic subgroup and its Levi factor. Then $\Res_{L\subseteq P}^{\GL_{2(n+m)}}(f)$ is a positive-definite function on $L^F$.
\end{lemma}
\begin{proof}
Let $i:P^F\to \GL_{2(n+m)}^F$ denote the inclusion map and $pr:P^F\to L^F$ denote the canonical projection. Then
\begin{equation*}
    Res^G_{L\subseteq P}(f)= pr_*i^*(f),
\end{equation*}
and so if $f$ is a positive-definite function, then by Lemma \ref{lem: pos-def push-pull} so is $ Res^G_{L\subseteq P}(f)$.
\end{proof}
\begin{corollary}
If $f$ is positive-definite bi-invariant on $L^F$, then $\Ind_{L\subseteq P}^G(f)$ is a positive-definite function.
\end{corollary}
\begin{proof}
Since parabolic induction and restriction are adjoints, if $f$ is a positive-definite bi-invariant function on $L^F$, then for any spherical function $\phi$ on $G^F$
\begin{equation*}
    \langle \Ind _{L\subseteq P}^G(f),\phi\rangle=|\Sp_{2(n+m)}^F|^2|\Sp_{2n}^F|^{-2}|\Sp_{2m}^F|^{-2}\langle f,\Res_{L\subseteq P}^G \phi\rangle
\end{equation*}
is non-negative by expanding both arguments in terms of spherical functions. Thus, $\Ind_{L\subseteq P}^G(f)$ is positive-definite.
\end{proof}

\subsection{Proof of Theorems \ref{thm: positivity}}

\begin{proof}[Proof of Theorem \ref{thm: positivity}]
Fix $\mu,\nu,\lambda$ and let $|\mu|=m$ and $|\nu|=n$. Assume that $|\lambda|=m+n$ as otherwise $C_{\lambda/\mu}^\nu(q,t)=0$. When necessary, view these partitions as partition-valued functions $O(\mathbf{L})\to \mathcal{P}$ by taking $\lambda(\varphi)=\lambda$ if $\varphi$ is the trivial character and $0$ otherwise.

Note that
\begin{equation*}
    C_{\lambda/\mu}^\nu(q,q^2)=\langle P_{\lambda/\mu}, s_\nu\rangle= \frac{\langle J_\lambda,J_\mu \omega\omega_{q^2,q}s_\nu\rangle_{q,q^2}}{c_\lambda(q,q^2)c'_\mu(q,q^2)}
\end{equation*}
is a positive scalar multiple of $\langle \phi_\lambda,\phi_\mu\ast \chi_\nu\rangle$ by Proposition \ref{prop: ch of spherical func} and Theorem \ref{thm: multiplicativity}.

Now by Proposition \ref{prop: spherical function positivity}, it is enough to show that $\phi_\mu\ast \chi_\nu=\Ind^G_{L\subseteq P}(\phi_\mu\times \chi_\nu)$ is a positive-definite bi-invariant function on $\GLq{2(m+n)}$. But since $\phi_\mu\times \phi_\nu$ is positive-definite, then by the corollary so is $\Ind^G_{L\subseteq P}(\phi_\mu\times \chi_\nu)$.
\end{proof}

\begin{remark}
This proof gives a representation-theoretic interpretation of the coefficients in the Schur expansion of skew Macdonald polynomials (with $q$ an odd prime power) in terms of coefficients of spherical functions in the bi-invariant parabolic induction, similar to how the classical Littlewood-Richardson coefficients can be viewed as the multiplicities of irreducible representations in the Young induction.
\end{remark}
\begin{remark}
Let $J_\mu^\perp(q,q^2)$ denote the adjoint of multiplication by $J_\mu(q,q^2)$ with respect to the $\langle,\rangle_{q,q^2}$ inner product. Some computations for small partitions in Sage suggest that 
\begin{equation*}
    \frac{\langle J_\mu^\perp(q,q^2)J_\lambda(q,q^2),s_\nu\rangle}{(1-q)^{|\lambda|+|\mu|}}\in\mathbf{N}[q]
\end{equation*}
extending the conjecture of Haglund. It is unclear if this is the correct normalization although replacing $J_\mu^\perp(q,q^2)$ by the normalized version (that takes $J_\mu(q,q^2)$ to $1$) leads to rational and not polynomial expressions. If this conjecture holds, it would be interesting to see what combinatorial interpretation the coefficients might have.
\end{remark}

\subsection{A vanishing theorem}
The characteristic map also gives a condition for vanishing of $C_{\lambda/\mu}^\nu(q,q^2)$. In the special case that $\mu=0$, the theorem states that $s_\nu$ appears in $J_{\lambda}(q,q^2)$ only if $s_{\lambda\cup\lambda}$ appears in the expansion of $s_\nu s_\nu$.

\begin{theorem}
\label{thm: vanishing}
If $\langle s_{\lambda\cup\lambda},s_{\mu\cup\mu}s_\nu s_\nu\rangle=0$, then $C_{\lambda/\mu}^\nu(q,q^2)$ vanishes as a function of $q$.
\end{theorem}

\begin{proof}
Suppose that $\langle s_{\lambda\cup\lambda},s_{\mu\cup\mu}s_\nu s_\nu\rangle=0$. Take $|\mu|=m$, $|\nu|=n$ and assume $|\lambda|=m+n$ as otherwise $C_{\lambda/\mu}^\nu(q,q^2)=0$. Then as in the proof of Theorem \ref{thm: positivity}, it is equivalent to show that $\langle \phi_\lambda,\phi_\mu\ast \chi_\nu\rangle=0$ since it's a positive scalar multiple of $C_{\lambda/\mu}^\nu(q,q^2)$, again interpreting partitions as partition-valued functions in the same way as above.

Let $\Av_{k}$ denote the averaging operation sending a function $f$ on $\GLq{2k}$ to $x\mapsto |H^F|^{-2}\sum _{h,h'\in H^F}f(hxh')$ where the group $H^F$ should be determined from context (either $\Spq{2k}$ or $\GLq{k}$). Let $G=\GL_{2(n+m)}$, $H=\Sp_{2(n+m)}$ and $L\subseteq P$ denote the Levi and parabolic defining $\ast$. Then note $\langle \phi_\lambda,\phi_\mu\ast \chi_\nu\rangle$ is a non-zero multiple of
\begin{equation*}
\begin{split}
    &\langle \Res_{L\subseteq P}^G(\Av_{n+m}(\chi_{\lambda\cup\lambda})),\Av_n\times \Av_m(\chi_{\mu\cup\mu}\times \chi_{\nu}\times \chi_\nu)\rangle
    \\=&\langle pr_*i^*\Av_{n+m}(\chi_{\lambda\cup\lambda}),\chi_{\mu\cup\mu}\times \chi_{\nu}\times \chi_\nu\rangle
    \\=&\langle \Av_{n+m}(\chi_{\lambda\cup\lambda}),i_*pr^*(\chi_{\mu\cup\mu}\times \chi_{\nu}\times \chi_\nu)\rangle
    \\=&\Av_{n+m}(\chi_{\lambda\cup\lambda}\cdot i_* pr^*(\chi_{\mu\cup\mu}\times \chi_{\nu}\times \chi_\nu))(1)
\end{split}
\end{equation*}
where in the last line the multiplication on functions is convolution.

Now 
\begin{equation*}
\begin{split}
    &\chi_{\lambda\cup\lambda}\cdot i_* pr^*(\chi_{\mu\cup\mu}\times \chi_{\nu}\times\chi_\nu)(x)
    \\=&\Tr\left(\rho_{\lambda\cup\lambda}(x)\sum _{y\in P^F}\rho_{\lambda\cup\lambda}(y^{-1})pr^*(\chi_{\mu\cup\mu}\times \chi_{\nu}\times\chi_\nu)(y)\right)
\end{split}
\end{equation*}
where $\rho_{\lambda\cup\lambda}$ denotes the corresponding representation and this is $0$ unless
\begin{equation*}
    \langle \chi_{\lambda\cup\lambda},pr^*(\chi_{\mu\cup\mu}\times \chi_{\nu}\times \chi_\nu)\rangle_{P^F}\neq 0,
\end{equation*}
because $pr^*(\chi_{\mu\cup\mu}\times \chi_{\nu}\times\chi_\nu)$ is an irreducible character. By Frobenius reciprocity
\begin{equation*}
    \langle \chi_{\lambda\cup\lambda},\Ind_{L\subseteq P}^G(\chi_{\mu\cup\mu}\times \chi_{\nu}\times \chi_\nu)\rangle\neq 0,
\end{equation*}
where $\Ind_{L\subseteq P}^G$ is the standard parabolic induction for $\GL_n$. This happens exactly when $\langle s_{\lambda\cup\lambda},s_{\mu\cup\mu}s_\nu s_\nu\rangle$ is non-zero.

Here, the theorem can be extended to hold for all $q$ because $C_{\lambda/\mu}^\nu(q,q^2)$ is a rational function in $q$.
\end{proof}

\section{Computation of Spherical Function Values}
\label{sec:computation of spherical function values}
In this section, the values of spherical functions on the double coset of non-symplectic transvections is computed. Similar computations could be done for the double cosets generated by $\mathrm{diag}(a,1,\dotsc,1)$ for $a\in \mathbf{F}_q^*$. This section should be seen as an application of the characteristic map to use the Pieri rule for Macdonald polynomials to compute spherical function values.

\begin{proposition}[{\cite[VI, \S 6]{M95}}]
Let 
\begin{equation*}
\psi'_{\lambda/\mu}:=\prod _{s\in C_{\lambda/\mu}\setminus R_{\lambda/\mu}}\frac{b_\lambda(s;q,t)}{b_\mu(s;q,t)},
\end{equation*}
where
\begin{equation*}
b_\lambda(s;q,t):=\frac{1-q^{a(s)}t^{l(s)+1}}{1-q^{a(s)+1}t^{l(s)}},
\end{equation*}
and where $C_{\lambda/\mu}$ denotes the columns of $\lambda$ intersecting $\lambda/\mu$ and similarly $R_{\lambda/\mu}$ but for rows. Then
\begin{equation*}
P_\mu(x;q,t)e_r(x)=\sum _{\lambda}\psi'_{\lambda/\mu}P_\lambda(x;q,t),
\end{equation*}
where the sum is over partitions $\lambda$ such that $\lambda\setminus \mu$ is a vertical strip with $r$ boxes.
\end{proposition}

First the value at the identity will be computed as a similar computation shows up in the transvection computation, even though the value is already known to be $1$.

Define $\delta$ as the specialization homomorphism on $\otimes \Lambda$ given by
\begin{equation*}
\delta(p_n(\varphi))=\frac{1}{q_\varphi^n-1}.
\end{equation*}
The following lemma is essentially proven in \cite[IV, \S 6]{M95}.
\begin{lemma}
\label{lemma: identity value}
For any $F\in \otimes \Lambda$, $\langle F,e_n(f_1)\rangle=\delta(\omega\omega_{q,q^2}F)$.
\end{lemma}
\begin{proof}
Since $\langle F,e_n(f_1)\rangle=\langle \omega\omega_{q,q^2} F,e_n(f_1)\rangle _{\GL_n}$ and $\omega\omega_{q,q^2}$ is invertible it suffices to check that $\langle F,e_n(f_1)\rangle_{\GL_n}=\delta(F)$. Since both sides are linear in $F$, it is enough to check on a basis, which is done in \cite[IV, \S 6]{M95}.
\end{proof}

The value of $\phi_\lambda(1)$ can be computed as follows. First note that $P_\lambda(x;q^{-2})=e_n$ when $\lambda=(1^n)$ and so
\begin{equation*}
\phi_\lambda(1)=|H^F|^{-1}\langle \phi_\lambda, I_{H^F}\rangle
\end{equation*}
can be computed by using the characteristic map and Lemma \ref{lemma: isometry}, giving
\begin{equation*}
\begin{split}
&(-1)^{|\lambda|}q^{-n}|\Spq{2n}|\psi_n(q^2)^{-1}q^{-(n^2-n)}\prod _{\varphi\in O(\mathbf{L})}q_\varphi^{-n(\lambda(\varphi)')}\langle J_{\lambda} (q,q^2),e_n(f_1)\rangle
\\=&(-1)^{|\lambda|}\prod _{\varphi\in O(\mathbf{L})}q_\varphi^{-n(\lambda(\varphi)')}\langle J_{\lambda}(q,q^2),e_n(f_1)\rangle.
\end{split}
\end{equation*}
Using Lemma \ref{lemma: identity value} and the fact that $\omega_{q,q^2}J_\lambda(q,q^2)=J_{\lambda'}(q^2,q)$ along with $\delta(\omega J_{\lambda'}(q^2,q))=(-1)^{|\lambda|}\prod _{s\in\lambda(\varphi)}q_\varphi^{a'(s)}$ \cite[VI, \S 8]{M95}, the inner product can be computed giving
\begin{equation*}
\begin{split}
&\prod _{\varphi\in O(\mathbf{L})}q_\varphi^{-n(\lambda(\varphi)')}\delta(J_{\lambda'}(q^2,q))
\\=&\prod _{\varphi\in O(\mathbf{L})}q_\varphi^{-n(\lambda(\varphi)')} \prod _{s\in\lambda(\varphi)}q_\varphi^{a'(s)}
\\=&1
\end{split}
\end{equation*}

The analogous computation of $\phi_\lambda(I_{H^Fg_\mu H^F})$ where $\mu(f_1)=(21^{n-2})$ and $0$ otherwise requires an additional lemma.

\begin{lemma}
\label{lemma: Pieri Expansion}
We have 
\begin{equation*}
    e_{n-1}(f_1)e_1(f_1)=\sum _{\|\lambda\|=n} \sum _{\lambda_0}\frac{q^{-n(\lambda_0')}c'_\lambda(q,q^2)}{c'_{\lambda_0}(q,q^2)(1-q)}\psi'_{\lambda/\lambda_0}\widehat{J}_{\lambda}(q,q^2),
\end{equation*}
where the second sum is over all partition-valued functions with $\|\lambda_0\|=n-1$ obtained by removing one box from some $\lambda(\varphi)$ with $d(\varphi)=1$, and $\widehat{J}_\lambda(q,q^2)$ denotes the dual basis to $J_\lambda(q,q^2)$ under the inner product.
\end{lemma}
\begin{proof}
First, note that $e_k(f_1)=\sum _{\|\lambda\|=k}(-1)^{|\lambda|}\delta(\omega J_{\lambda'}(q^2,q))\widehat{J_{\lambda}}(q,q^2)$ which is an easy consequence of Lemma \ref{lemma: identity value}. Thus
\begin{equation*}
\begin{split}
&e_{n-1}(f_1)e_1(f_1)
\\=&\left(\sum_{\|\lambda_1\|=n-1} (-1)^{|\lambda_1|}\delta(\omega J_{\lambda_1'}(q^2,q))\widehat{J}_{\lambda_1}(q,q^2)\right)\left(\sum_{\|\lambda_2\|=1} (-1)^{|\lambda_2|}\delta(\omega J_{\lambda_2'}(q^2,q))\widehat{J}_{\lambda_2}(q,q^2)\right).
\end{split}
\end{equation*}

There are exactly $q-1$ partition valued functions $\lambda$ with $\|\lambda\|=1$, which give $e_1(\varphi)$ for $\varphi\in L_1$ as the polynomials $J_{\lambda_2}(q,q^2)$. Thus, apply Pieri's rule for $r=1$, and so $C_{\lambda/\mu}$ consists of the column that is added, and similarly for the row. The arm/leg lengths in $\lambda$ are exactly one more than in $\mu$ because of the added box, and so after relabeling $\lambda_1$ with $\lambda_0$
\begin{equation*}
e_{n-1}(f_1)e_1(f_1)=\sum _{\|\lambda\|=n}\sum _{\lambda_0}\frac{q^{-n(\lambda_0')}c'_\lambda(q,q^2)}{c'_{\lambda_1}(q,q^2)(1-q)}\psi'_{\lambda/\lambda_0}\widehat{J}_{\lambda}(q,q^2),
\end{equation*}
where $\delta(e_1(\varphi))=1$, and the remaining factors $c'_\lambda(q,q^2)$ and $(1-q)$ come from the scaling $\widehat{J}_\lambda(q,q^2)=c_\lambda'(q,q^2)^{-1}P_\lambda(q,q^2)$.
\end{proof}

Now the spherical functions of interest may be computed.

\begin{proposition}
Let $\mu:O(\mathbf{M})\rightarrow \mathcal{P}$ given by $\mu(f_1)=(21^{n-2})$ and $0$ otherwise. Then
\begin{equation*}
\phi_\lambda(g_\mu)=\frac{q^{2n-2}(q^2-1)}{(q^{2n}-1)(q^{2n-2}-1)}\left(\sum _{\lambda_0}\frac{c'_{\lambda}(q,q^2)\psi_{\lambda/\lambda_0}'}{c'_{\lambda_0}(q,q^2)(1-q)}q^{n(\lambda_0')-n(\lambda')}-\frac{q^{2n}-1}{q^{2n-2}(q^2-1)}\right)
\end{equation*}
for all spherical functions $\phi_\lambda$ where $\lambda_0$ is obtained from $\lambda$ by removing a single box from some $\lambda(\varphi)$ with $d(\varphi)=1$.
\end{proposition}
\begin{proof}
Note that
\begin{equation*}
\begin{split}
\phi_\lambda(g_\mu)&=\frac{1}{|H^Fg_\mu H^F|}\langle \phi_\lambda,I_{H^Fg_\mu H^F}\rangle
\\&=\frac{q^{2n-2}(q^2-1)}{(q^{2n}-1)(q^{2n-2}-1)}\left\langle (-1)^{|\lambda|}\prod _{\varphi\in O(\mathbf{L})}q_\varphi^{-n(\lambda(\varphi)')} J_{\lambda}(q,q^2),P_{\mu(f_1)}(f_1;q^{-2})\right\rangle
\end{split}
\end{equation*}
by Lemma \ref{lemma: isometry}.

From \cite[Eq. 2.4]{H92}, 
\begin{equation*}
    P_{\mu(f_1)}(f_1;q)=e_{n-1}(f_1)e_1(f_1)-\left(\sum _{i=0}^{n-1}q^{i}\right)e_n(f_1).
\end{equation*}
Lemma \ref{lemma: identity value} can handle the $e_n(f_1)$ term, and Lemma \ref{lemma: Pieri Expansion} will give the other term. Note that although the notation $c_\lambda'$, and other functions indexed by partitions, is used for partition-valued functions by taking a product over the domain, in almost all cases due to cancellation only one partition will be relevant.

Now evaluate the desired inner product, obtaining
\begin{equation*}
\begin{split}
&\langle (-1)^{|\lambda|} J_{\lambda}(q,q^2),P_{\mu(f_1)}(f_1,q^{-2})\rangle
\\=&\sum _{\lambda_0}\frac{q^{-n(\lambda_0')}c'_{\lambda}(q,q^2)\psi_{\lambda/\lambda_0}'}{c'_{\lambda_0}(q,q^2)(1-q)}-\frac{q^{2n}-1}{q^{2n-2}(q^2-1)}q^{-n(\lambda')}
\end{split}
\end{equation*}
which gives the desired result with the remaining factors since 
\begin{equation*}
\prod _{\varphi\in O(\mathbf{L})}q_\varphi^{n(\lambda_0(\varphi)')}=\delta(J_{\lambda_0'}(q^2,q)).
\end{equation*}
\end{proof}

\section{Character sheaves on symmetric spaces}
\label{sec: technical results}
Here, the bi-invariant parabolic induction defined in Section \ref{sec: induction for functions} is connected to the theory developed by Grojnowski \cite{G92} and Henderson \cite{H01t} regarding character sheaves on symmetric spaces. References are also given to the work of Shoji and Sorlin \cite{SS13,SS14a,SS14b} where possible for any cited results as both \cite{G92} and \cite{H01t} are unpublished. Some familiarity with the theory of $l$-adic sheaves and algebraic groups is assumed, although beyond the results of \cite{H01t}, only some basic formal properties will be needed to translate the results into the language of functions.

\subsection{Motivation}
For $\GLq{n}$, most properties of the characteristic map are more easily seen when viewing the domain not as the ring of class functions but rather the representation ring. In the symmetric space setting, the relationship between the spherical functions and representations is less well-behaved, so representations should be replaced by sheaves instead. 

All operations that can be done on representations have analogues for sheaves (for instance, tensor products, induction) but there are additional operations available. Moreover, the relationship between these sheaves and functions is well-understood and so anything that can be proven for sheaves has a function analogue. Working with sheaves, the usual notions of parabolic induction and restriction for groups also have a natural analogue.

\subsection{Preliminaries on $l$-adic sheaves}
This section mostly follows \cite{H01t}, see there for further details and references.

Fix some prime $l$ not equal to the characteristic of $\mathbf{F}_q$ (nothing will depend on the choice of $l$). Since only algebraic numbers appear, all statements in this section about $\overline{\mathbf{Q}_l}$ functions can be readily transferred to $\mathbf{C}$.

All varieties will be defined over $\overline{\mathbf{F}_q}$. An $\mathbf{F}_q$-structure on a variety $X$ is given by a \emph{Frobenius map} $F:X\to X$. A morphism of varieties with $\mathbf{F}_q$ structures is defined over $\mathbf{F}_q$ if it is $F$-equivariant. Let $X^F$ denote the corresponding variety over $\mathbf{F}_q$. 

If $X$ is a variety, let $D(X)$ denote the bounded derived category of constructible $\overline{\mathbf{Q}_l}$-sheaves of finite rank on $X$. The objects of this category are called complexes.

Say that a complex $K$ is $F$-stable if $F^*K\cong K$. Now to any $F$-stable complex $K$, and a choice of isomorphism $\phi:F^*K\to K$, there is an associated function $\chi_{K}: X^F\to \overline{\mathbf{Q}_l}$ called the \emph{characteristic function}. It is defined by
\begin{equation*}
    \chi_{K}(x):=\sum _{i}(-1)^i \Tr(\phi,\mathcal{H}^i_xK),
\end{equation*}
where $\mathcal{H}_x^iK$ denotes the stalk at $x$ of the cohomology sheaf of $K$. The dependence on $\phi$ will be dropped in the notation but a choice of such an isomorphism must always be made when taking the characteristic function.

If $f:X\to Y$ is a morphism of varieties, there are two functors, $f^*:D(Y)\to D(X)$, called the \emph{pullback} and $f_!:D(X)\to D(Y)$, called the \emph{compactly supported pushforward}.

There is a shift functor $[n]:D(X)\to D(X)$ taking a complex $K$ to the complex $K[n]$ with $K[n]_m=K_{n+m}$ and this functor commutes with pushforward and pullback and the corresponding effect on the characteristic function is multiplication by $(-1)^n$.

Finally, direct sums of complexes in $D(X)$ correspond to addition and the tensor product of complexes in $D(X)$ corresponds to multiplication of the corresponding characteristic functions.

If $G$ is an algebraic group acting on $X$, let $D^G(X)$ denote the $G$-equivariant derived category. All the constructions above have analogues in the equivariant setting. In particular, if the action is defined over $\mathbf{F}_q$ then the characteristic function of an equivariant sheaf defined over $\mathbf{F}_q$ is invariant under the action of $G^F$.

If $f:X\to Y$ is a principal $G$-bundle, then there is an equivalence of categories $f^*:D(Y)\to D^G(X)$ (whose composition with the forgetful functor $D^G(X)\to D(X)$ gives the usual pullback) and the inverse is denoted by $f_\flat: D^G(X)\to D(Y)$. If $f$ is defined over $\mathbf{F}_q$ then the characteristic functions satisfy
\begin{equation*}
    \chi_{f_\flat K}(y)=|G^F|^{-1}\sum _{f(x)=y, x\in X^F}\chi_K(x).
\end{equation*}
If $H\subseteq G$, then this equivalence also restricts to an equivalence $D^{H\times H}(G)\cong D^H(G/H)$ between the $H$ bi-equivariant sheaves on $G$ and $H$ equivariant sheaves on $G/H$ and so when convenient bi-equivariant sheaves can be thought to live on the symmetric space $G/H$ (this is like how there is no difference between working with bi-invariant functions on $G$ and invariant functions on $G/H$).

\subsection{Induction functors}
Here, the induction functors analogous to Deligne-Lusztig induction for $\GL_n$ are introduced. They were first introduced by Grojnowski \cite{G92} but the definition used follows that of Henderson \cite{H01t} which was extended in the work of Shoji and Sorlin \cite{SS14b}.

It is important to note that these functors do not preserve any sort of $\mathbf{F}_q$ structure unless the parabolic subgroup used is rational. Even if $P$ is rational, to define induction for functions some choice of $\mathbf{F}_q$ structure must be placed on the sheaves themselves. For now, the definition does not involve any $\mathbf{F}_q$ structure whatsoever.

For any $\iota$-stable Levi subgroup $L$ with $P=LU$ an $\iota$-stable parabolic, define a parabolic induction functor $\Ind_{L\subseteq P}^G:D^{H_L\times H_L}(L)\to D^{H\times H}(G)$ taking bi-equivariant sheaves on $L$ to $G$ as follows. Consider the diagram
\begin{equation*}
\begin{tikzcd}
L&H\times P\times H\arrow[l,"pr",swap]\arrow[r,"q"]&H\times _{H_P}P\times _{H_P}H\arrow[r,"i"]&G
\end{tikzcd}
\end{equation*}
where $pr(h,p,h')=\overline{p}$, $q$ is the quotient morphism and $i(h,p,h')=hph'$ and the action of $H_P\times H_P$ on $L$ is given by $(h,h')\cdot l=\overline{h}l\overline{h'}^{-1}$ with $H\times H$ acting trivially, the action of $H_P\times H_P$ on $H\times P\times H$ is by $(h_1,h_2)\cdot(h,p,h')=(hh_1^{-1},hph_2^{-1},h_2h')$ and the action of $H\times H$ is by $(h_1,h_2)\cdot (h,p,h')=(h_1h,p,h'h_2^{-1})$, and $H\times H$ acts on the last two spaces on the left and right similarly.

Then all the maps are equivariant, with $pr$ being $H_P\times H_P$ equivariant as $P$ is $\iota$-stable, so define the induction functor $\Ind_{L\subseteq P}^G:D^{H_L\times H_L}(L)\to D^{H\times H}(G)$ by
\begin{equation*}
    \Ind_{L\subseteq P}^G:=i_! q_\flat pr^*[\dim U+2\dim H/H_P].
\end{equation*}

A key result is that the composition of two induction functors is again an induction functor. The following proposition is given in \cite[Prop. 2.19]{H01t} and also in \cite[Prop. 4.3]{SS14b} and the proof is similar to the usual proof of transitivity of induction for groups.

\begin{proposition}[{\cite[Prop. 2.19]{H01t}}]
\label{prop: transitivity of induction}
Let $M\subseteq L$ be $\iota$-stable Levi subgroups, and let $Q\subseteq P$ be $\iota$-stable parabolic subgroups with Levi factors $M$ and $L$ respectively. Then
\begin{equation*}
    \Ind_{L\subseteq P}^G \circ \Ind_{M\subseteq Q\cap L}^L\cong \Ind_{M\subseteq Q}^G.
\end{equation*}
\end{proposition}

The following proposition follows easily from the definition of induction.
\begin{proposition}
\label{prop: Ind is same}
Let $L$ be a rational $\iota$-stable Levi subgroup, with a rational parabolic $P=LU$. Then if $K$ is a complex on $L$ with an $\mathbf{F}_q$ structure, and $\Ind_{L\subseteq P}^G(K)$ is given the induced $\mathbf{F}_q$ structure,
\begin{equation*}
    \chi_{\Ind_{L\subseteq P}^G(K)}=|H^F\cap P^F|^{-2}(-1)^{\dim U}\sum _{\substack{h,h'\in H^F\\ hxh'\in P^F}}\chi_K(\overline{hxh'}).
\end{equation*}
\end{proposition}

\subsection{Characteristic functions of induced complexes}
The induction functors define parabolic induction of complexes but give no rational structure on the resulting complex and so more is needed to define induction of functions. The idea is to restrict to a smaller set of sheaves which are closed under induction and have a canonical $\mathbf{F}_q$ structure. For more details on this section, see \cite[Ch. 5, 6]{H01t} or \cite{S16}, noting that any sheaf on $G/H\times V$ can be restricted to $G/H\times \{0\}$.

First, consider the group case. For any maximal torus $T$ and a tame rank one local system $\mathcal{L}$ on $T$, there is an associated complex $K_{(T,\mathcal{L})}$ defined using intersection cohomology. This complex depends only on the $G$ orbit of $(T,\mathcal{L})$.

Moreover, if $T$ is rational and $\mathcal{L}$ is $F$-stable, there is a unique $\mathbf{F}_q$ structure making the characteristic function an irreducible character of $T^F$. Then $K_{(T,\mathcal{L})}$ has an induced $\mathbf{F}_q$ structure coming from the structures on $T$ and $\mathcal{L}$ (through the construction in terms of intersection cohomology). The corresponding characteristic functions $\chi_{K_{(T,\mathcal{L})}}$ depend only on the $G^F$ orbit of $(T,\mathcal{L})$. 

These complexes $K_{(T,\mathcal{L})}$ are exactly the ones obtained by inducing $\mathcal{L}$ from $T$ with any choice of Borel subgroup. The benefit of the intersection cohomology definition is that it gives an $\mathbf{F}_q$ structure.

In the symmetric space setting where $G=\GL_{2n}$ and $H=\Sp_{2n}$, the above still holds with modifications, namely replacing a maximal torus with a maximal $\iota$-stable torus and characters replaced with spherical functions. There is only one $H$-conjugacy class of maximal $\iota$-stable tori, and so in particular every maximal $\iota$-stable torus $T$ contains a maximal $\iota$-split torus $T^{-\iota}$ with $T^{-\iota}\cong T/T_H$ and is contained in an $\iota$-stable Borel subgroup. The $H^F$ conjugacy classes of maximal $\iota$-split tori are in bijection with the $\GL_n^F$-conjugacy classes of $F$-stable maximal tori of $\GL_n$, as every maximal $\iota$-split torus can be conjugated to lie in $\GL_n\times \GL_n$ (acting on the first and last $n$ coordinates) where it acts by $(t,t)$.

There are complexes $K_{(T^{-\iota},\mathcal{L})}$ associated to tame rank one local systems $\mathcal{L}$ on maximal $\iota$-split tori $T^{-\iota}$. When both have an $\mathbf{F}_q$ structure, then $K_{(T^{-\iota},\mathcal{L})}$ does as well and their characteristic functions are invariant under $H^F$. See \cite{H01t} or \cite{S16} for details on their construction and properties.

The transitivity of induction implies that the collection of complexes $\Ind_{T\subseteq B}^G(\mathcal{L})$ is closed under bi-invariant parabolic induction. The method of defining parabolic induction for functions through the functor will be to give these complexes an $\mathbf{F}_q$ structure, and then show that their characteristic functions form a basis for the bi-invariant functions, uniquely defining induction for any bi-invariant function. The next theorem relates the induced complexes to the complexes $K_{(T^{-\iota},\mathcal{L})}$ and is due to Grojnowski \cite{G92}, although see also \cite[Theorem 1.16]{SS14a} or \cite[Theorem 5.5]{H01t} where a more complete proof is given.

\begin{theorem}[{\cite[Lemma 7.4.4]{G92}}]
\label{thm: DL sheaf iso}
Let $T$ be a rational maximal $\iota$-stable torus and $\mathcal{L}$ an $F$-stable tame rank one local system on $T^{-\iota}$. There are isomorphisms
\begin{equation*}
    \Ind_{T\subseteq B}^G(\mathcal{L})\cong\Ind_{L_0\subseteq P_0}^G(K^{L_0}_{(T^{-\iota},\mathcal{L})})\cong  K_{(T^{-\iota},\mathcal{L})}\otimes H_c^\bullet (\mathcal{B}^{Z_H(T^{-\iota})})[\dim T^{-\iota}].
\end{equation*}
where $\mathcal{B}^{Z_H(T^{-\iota})}$ denotes the flag variety of $Z_H(T^{-\iota})$ and the second is defined over $\mathbf{F}_q$.
\end{theorem}
This theorem means that $\Ind_{T\subseteq B}^G(\mathcal{L})$ has a canonical $\mathbf{F}_q$ structure and so the characteristic function of induced complexes, at least from a maximal torus, may be taken. Also, since induction from a maximal torus does not depend on the Borel chosen, the notation $\Ind_{T\subseteq B}^G$ will be used even when a particular Borel is not specified.

The characteristic functions of the complexes $K_{(T,\mathcal{L})}$ are related to the Deligne-Lusztig characters. These functions are essentially the basic functions already defined.

\begin{proposition}
Let $T$ be a rational maximal $\iota$-stable torus and $\theta$ an irreducible character of $(T^{-\iota})^F$. Let $\mathcal{L}_\theta$ denote the corresponding local system on $T^{-\iota}$. Then
\begin{equation*}
\begin{split}
    (-1)^n\chi_{K_{(T^{-\iota},\mathcal{L}_\theta)}}&=\frac{|(T^{-\iota})^F|}{|(T^{-\iota})^F|_{q\mapsto q^2}}\chi_{\Ind_{L_0\subseteq P_0}^G(K_{(T^{-\iota},\mathcal{L}_\theta)})}
    \\&=\zeta_{T^{-\iota}}^{\GL_{n}}(\cdot|\theta)_{q\mapsto q^2}.
\end{split}
\end{equation*}
\end{proposition}
The first equality follows from the previous theorem by taking characteristic functions and the second equality follows from \cite[Proposition 6.9]{H01t} or Theorem 5.3.2 in \cite{BKS90}. As a corollary, this also shows that the $\chi_{K_{(T^{-\iota},\mathcal{L}_\theta)}}$ form a basis for the space of bi-invariant functions on $G^F$, since the basic functions do.

This proposition give a way to define induction for any bi-invariant function. Define
\begin{equation*}
    \Ind_{L\subseteq P}^G\chi_{\Ind_{T\subseteq B}^L(\mathcal{L})}=\chi_{\Ind_{T\subseteq B'}^G(\mathcal{L})}
\end{equation*}
where $B'$ is any Borel subgroup with Levi factor $T$ in $G$ and extend by linearity. Since the $\chi_{\Ind_{T\subseteq B}^L(\mathcal{L})}$ form a basis for the $H_L^F$ bi-invariant functions on $L^F$, this is well-defined. 
In the case that the parabolic subgroup $P$ is rational, the induction functor actually gives an $\mathbf{F}_q$ structure on $\Ind_{L\subseteq P}^G(K)$ if $K$ has one. 

Finally, it is shown that the $\mathbf{F}_q$-structure coming from induction through a rational parabolic subgroup agrees with the canonical one when applied to $K_{(T^{-\iota},\mathcal{L})}$.

\begin{proposition}
Let $L$ be a rational $\iota$-stable Levi subgroup and $P$ a rational $\iota$-stable parabolic with $L$ as its Levi factor. Let $T^{-\iota}$ be a rational maximal $\iota$-split torus and an $F$-stable tame rank one local system, and $\theta$ the character of $T^{-\iota}$ associated to $\mathcal{L}$. Then if $\Ind_{T\subseteq B}^L(\mathcal{L})$ and $\Ind_{T\subseteq B'}^G(\mathcal{L})$ are given their canonical rational structures,
\begin{equation*}
    \Ind_{L\subseteq P}^G(\Ind_{T\subseteq B}^L(\mathcal{L}))= \Ind_{T\subseteq B'}^G(\mathcal{L})
\end{equation*}
where $B$ and $B'$ are any $\iota$-stable Borel subgroups and this is an isomorphism over $\mathbf{F}_q$.
\end{proposition}

\begin{proof}
First, note that it suffices to prove the proposition for Levi subgroups of the form $L\cong \GL_{2n}\times\GL_m\times \GL_m$ where $H_L\cong \Sp_{2n}\times \GL_m$. This is because all $\iota$-stable Levi subgroups are of the form $\GL_{2n}\times \prod \GL_{n_i}\times\GL_{n_i}$ and $H_L$ are of the form $\Sp_{2n}\times \prod \GL_{n_i}$ (see \cite[5.6]{SS13} for example) and all factors except the $\GL_{2n}/\Sp_{2n}$ factor are isomorphic to the symmetric space $\GL_n\times \GL_n/\GL_n\cong \GL_n$, where the proposition is already known (the induction operation is equivalent to Deligne-Lusztig induction up to the twist and $\chi_{K_{(T,\mathcal{L})}}$ is simply a Deligne-Lusztig character \cite[Proposition 9.2]{L90}).

Now let $L_0\subseteq G$ be a subgroup of the form $\GL_{n+m}\times \GL_{n+m}$ with $L_0\cap L\subseteq L$ a subgroup of the form $\GL_n\times \GL_n\times \GL_m\times \GL_m$ with $H_{L_0}\cong \GL_n\times \GL_m$. Then the proposition for $L_0$ is already given by Theorem \ref{thm: DL sheaf iso}.

It's clear from the proof of Proposition \ref{prop: transitivity of induction} that if all Levi and parabolic subgroups involved are rational, the isomorphism is defined over $\mathbf{F}_q$ (the isomorphism is constructed by chasing a diagram which will contain only varieties defined over $\mathbf{F}_q$) and so this proposition can be reduced to the known case as
\begin{equation*}
\begin{split}
    \Ind_{L\subseteq P}^G(\Ind_{T\subseteq B}^L(\mathcal{L}))&\cong \Ind_{L_0\cap L\subseteq P'}^G(\Ind_{T\subseteq B''}^{L_0\cap L}(\mathcal{L}))
    \\&\cong \Ind_{L_0\subseteq P_0}^G(\Ind_{T\subseteq B''}^{L_0}(\mathcal{L}))
    \\&\cong\Ind_{T\subseteq B'}^G(\mathcal{L})
\end{split}
\end{equation*}
where $P'$ is a rational $\iota$-stable parabolic contained in $P$ and $P_0$ is a rational $\iota$-stable parabolic chosen to contain $P'$ (and all isomorphisms are over $\mathbf{F}_q$).
\end{proof}

The key result of Section \ref{sec: induction for functions}, Proposition \ref{prop:ind for functions}, follows by taking characteristic functions of both sides.

\section*{Acknowledgements}
This research was supported in part by NSERC. The author would like to thank Anthony Henderson for some comments on an earlier draft and for pointing out some references, Arun Ram and Dan Bump for some insightful suggestions and Cheng-Chiang Tsai for clarifying some aspects of character sheaves, as well as Persi Diaconis, and Aaron Landesman for helpful discussions.

\bibliography{bibliography}{}
\bibliographystyle{amsplain}

\end{document}